\documentclass[a4paper,leqno]{amsart}

\usepackage{latexsym}
\usepackage[english]{babel}
\usepackage{fancyhdr}
\usepackage[mathscr]{eucal}
\usepackage{amsmath}
\usepackage{mathrsfs}
\usepackage{amsthm}
\usepackage{amsfonts}
\usepackage{amssymb}
\usepackage{amscd}
\usepackage{bbm}
\usepackage{graphicx}
\usepackage{subcaption}
\usepackage{graphics}
\usepackage{latexsym}
\usepackage{color}

\usepackage{pifont}
\usepackage{booktabs} % for '\midrule' macro

\newcommand{\ud}{\mathrm{d}}

\newcommand{\ii}{\mathrm{i}}
\newcommand{\cH}{\mathcal{H}}

\newcommand{\C}{\mathbb C}

\newcommand{\N}{\mathbb N}
\newcommand{\R}{\mathbb R}

\theoremstyle{plain}
\newtheorem{theorem}{Theorem}[section]
\newtheorem{lemma}[theorem]{Lemma}
\newtheorem{corollary}[theorem]{Corollary}
\newtheorem{proposition}[theorem]{Proposition}

\theoremstyle{definition}

\newtheorem{remark}[theorem]{Remark}

\numberwithin{equation}{section}

\begin{document}

\title[2d singular Hartree equation]
{Standing waves and global well-posedness for the 2d Hartree equation with a point interaction}
\author[V.~Georgiev]{Vladimir Georgiev}
\address[V.~Georgiev]{Department of Mathematics, University of Pisa \\ largo Bruno Pontecorvo 5 \\ I-56127 Pisa (ITALY) \\ and Faculty of Science and Engineering, Waseda University \\
3-4-1, Okubo, Shinjuku-ku, Tokyo 169-8555 (JAPAN) \\ and  Institute of Mathematics and Informatics at Bulgarian Academy of Sciences, Acad. Georgi Bonchev Str., Block 8, 1113 Sofia (BULGARIA)}
\email{georgiev@dm.unipi.it}
\author[A.~Michelangeli]{Alessandro Michelangeli}
\address[A.~Michelangeli]{Institute for Applied Mathematics, and Hausdorff Center for Mathematics, University of Bonn \\ Endenicher Allee 60 \\ D-53115 Bonn (GERMANY) \\ and TQT Trieste Institute for Theoretical Quantum Technologies, Trieste (ITALY)}
\email{michelangeli@iam.uni-bonn.de}
\author[R.~Scandone]{Raffaele Scandone}
\address[R.~Scandone]{Gran Sasso Science Institute, viale Francesco Crispi 7 \\ I-67100 L'Aquila (ITALY)}
\email{raffaele.scandone@gssi.it}

%\dedicatory{}

\begin{abstract}
We study a class of two-dimensional non-linear Schr\"odinger equations with point-like singular perturbation and Hartree non-linearity. The point-like singular perturbation of the free Laplacian induces appropriate perturbed Sobolev spaces that are necessary for the study of ground states and evolution flow. We include in our treatment both mass sub-critical and mass critical Hartree non-linearities. Our analysis is two-fold: we establish existence, symmetry, and regularity of ground states, and we demonstrate the well-posedness of the associated Cauchy problem in the singular perturbed energy space. The first goal, unlike other treatments emerging in parallel with the present work, is achieved by a non-trivial adaptation of the standard properties of Schwartz symmetrisation for the modified Weinstein functional. This produces, among others, modified Gagliardo-Nirenberg type inequalities that allow to efficiently control the non-linearity and obtain well-posedness by energy methods. The evolution flow is proved to be global in time in the defocusing case, and in the focusing and mass sub-critical case. It is also global in the focusing and mass critical case, for initial data that are suitably small in terms of the best Gagliardo-Nirenberg constant.
\end{abstract}

\date{\today}



\thanks{Partially supported by the Italian National Institute for Higher Mathematics -- INdAM (V.G., A.M., R.S.),
 the project ``Problemi stazionari e di evoluzione nelle equazioni di campo non-lineari dispersive'' of GNAMPA -- Gruppo Nazionale per l'Analisi Matematica (V.G.), the PRIN project no.~2020XB3EFL of the MIUR -- Italian Ministry of University and Research (V.G.), the Institute of Mathematics and Informatics at the Bulgarian Academy of Sciences (V.G.), the Top Global University Project at Waseda University (V.G.), and the Alexander von Humboldt Foundation (A.M.). The last two authors gratefully acknowledge the kind hospitality of V.G.~at the Department of Mathematics of the University of Pisa, where a large part of this work was carried on.
}

\maketitle

%\tableofcontents

\section{Introduction and main results}\label{sec:intro-main}

In this work we study the existence of standing waves and the well-posedness of the two-dimensional, point-like perturbed, singular Hartree equation
\begin{equation}\label{singular_hartree}
\ii\partial_t\psi=-\Delta_{\alpha}\psi+(w*|\psi|^2)\psi\,,
\end{equation}
in the complex-valued unknown $\psi\equiv \psi(t,x)$, $(t,x)\in[0,\infty)\times \mathbb{R}^2$, where the convolution potential $w$ is a real-valued measurable function, and $-\Delta_{\alpha}$ is the self-adjoint point-like singular perturbation of the ordinary negative Laplacian on $L^2(\mathbb{R}^2)$-functions, located at the origin $x=0$ and with inverse scattering length given by $\alpha\in\mathbb{R}$, in suitable units.

%%%%%%%%5 the parameter $\theta=\pm 1$ selects, respectively, the defocusing or focusing behaviour of the non-linearity,

There are multiple motivations for the recently increasing interest towards non-linear Schr\"{o}\-dinger equations (NLS) of Hartree type like \eqref{singular_hartree} (i.e., convolutive, non-local), or also with local semi-linearity, with the additional point-like perturbation at one or more distinguished points. On the one hand, such equations are realistic effective models for the time evolution of the density of particles $|\psi(t,x)|^2$ of a large assembly of identical bosons at ultra-low temperature interacting through a two-body potential $w$ and additionally coupled to strong, delta-like impurities located at fixed points in space -- one impurity at $x=0$ in the present case. Without the singular perturbation, hence with the Laplacian $-\Delta$ in place of $-\Delta_\alpha$, or also with variants such as the magnetic or the semi-relativistic Laplacian, a variety of highly sophisticated mathematical techniques is today available to rigorously derive such effective NLS, in the limit of infinitely many particles, from the linear, many-body Schr\"{o}dinger equation \cite{Pickl-LMP-2011,Benedikter-Porta-Schlein-2015}. For this apparatus to work, and the same would apply to the point-like, singular-perturbed Laplacian, \emph{it is crucial to know that the emerging NLS is well-posed} in a convenient regularity space, so as to have suitable a priori norms available that are uniformly conserved in time and can be exploited in the control of the limit of infinitely many particles.

From an analogous perspective, the non-perturbed version of \eqref{singular_hartree} has a meaning of effective equation for the quantum plasma dynamics, when considered with the magnetic Laplacian and as part of coupled systems of Maxwell-Schr\"{o}dinger type, a scenario already well under study both for its rigorous derivation as an effective description from the many-body Pauli-Fierz Hamiltonian \cite{Leopold-Pickl-MaxwSchr-2016}, and concerning well-posedness and standing waves \cite{Coclite-Georgiev-MSsolitary-2004,Nakamura-Wada-MaxwSchr_CMP2007,Ginibre-Velo-2008,Colin-Watanabe-2019,AntMarcScand-MaxSchr-2019}: here too, adding a point-like perturbation is of relevance for a more accurate modelling with impurities.

 On the other hand, there is an autonomous relevance of \eqref{singular_hartree} \emph{per se}, as \eqref{singular_hartree} poses novel technical problems arising in the study of its various features as a dynamical equation (well-posedness, scattering, ground state), precisely due to the singular point-like perturbed nature of $-\Delta_\alpha$, which makes an amount of standard analytic tools from the theory of the linear and non-linear Schr\"{o}dinger equation not directly applicable.

 Indeed, whereas \emph{in dimension one} the point-like nature of the singularity has the explicit structure $-\Delta+\delta(x)$, thereby allowing for a much more accurate analysis of \eqref{singular_hartree}, and its counterparts with local semi-linearity, in terms of local and global well-posedness, blow-up, scattering, asymptotic stability, solitons, standing waves, ground state \cite{Adami-Sacchetti-2005-1D-NLS-delta,DellaCasa-Sacchetti-2006,Adami_JPA2009_1D_NLS_with_delta,Adami-Noja-Visciglia-2013,Adami-Noja-2014,AnguloPava-Ferreira-2014,Banica-Visciglia-2016,Ikeda-Inui-2017,Ianni-LeCoz-Royer-2017,Cuccagna-Maeda-SIAM2019,AnguloPava-Hernandez-2019,Masaki-Murphy-Segata-2019,Masaki-Murphy-Segata-2020,Cuccagna-Maeda-DCDSsurvey2021}, the case of \emph{dimension two or three} presents a much more uncharted territory, ultimately due to the circumstance that $-\Delta_\alpha$ is not built as a form perturbation.

 This has led only recently to our characterisation \cite{Georgiev-M-Scandone-2016-2017} of a whole family of fractional powers of the three-dimensional $-\Delta_\alpha$ and hence of the new `perturbed' Sobolev spaces adapted to it, whence then the first proof of local and global well-posedness of the three-dimensional counterpart of \eqref{singular_hartree} \cite{MOS-SingularHartree-2017}. In the same spirit, the well-posedness of the two- and three-dimensional point-perturbed singular NLS with local semi-linearity has been later established in \cite{Caccia-Finco-Noja-deltaNLS-2020}, also outlining perspectives and open problems (on the types of blow-up and scattering). And it is only in the course of the final draft of this work that two significant contributions appeared on the issues of existence and stability or instability of the ground state of the two-dimensional point-perturbed singular NLS with local semi-linearity, respectively in the mass sub-critical \cite{Adami-Boni-Carlone-Tentarelli-2021} and the energy sub-critical \cite{Fukaya-Georgiev-Ikeda-2021} case. Also 	the linear flow generated by $-\Delta_\alpha$, in two dimensions, has attracted recent attention, significantly with the proof of the $L^p$-boundedness of the wave operator associated to the pair $(-\Delta_\alpha,-\Delta)$ \cite{CMY-2018-2Dwaveop,CMY-2018-2DwaveopERR,Yajima-2021-Lpbdd-delta-2D}, whence also Strichartz and dispersive estimates.

 In order to enter the technical aspects and present our main results, let us concisely recall (see, e.g., \cite[Chapter I.5]{albeverio-solvable}) that, for given $\alpha\in\mathbb{R}$, $-\Delta_{\alpha}$ is that self-adjoint extension on $L^2(\mathbb{R}^2)$ of the non-negative symmetric operator  $-\Delta|_{C^\infty_0(\mathbb{R}^2\setminus\{0\})}$ whose operator domain $\mathcal{D}(-\Delta_\alpha)$ and action are defined as
 \begin{eqnarray}
  {\;}\qquad\qquad \mathcal{D}(-\Delta_{\alpha} ) &=& \Big\{g\in L^2(\mathbb{R}^2)\,\Big|\,g=f_{\omega}+\frac{ f_{\omega}(0)}{\beta_\alpha(\omega)}\,\mathcal{G}_{\omega}\textrm{ with }f_{\omega}\in H^2(\mathbb{R}^2)\Big\}\,, \label{eq:op_dom} \\
  (\omega-\Delta_{\alpha})\,g&=&(\omega-\Delta)\,f_{\omega}\,, \label{eq:opaction}
 \end{eqnarray}
 where $\omega>0$ is an arbitrarily fixed constant,
 \begin{equation}\label{eq:Gomega-intro}
  \mathcal{G}_{\omega}(x) \;:=\; (2\pi)^{-2} \int_{\mathbb{R}^2} e^{-\mathrm{i} x\xi} \frac{\ud \xi}{\omega+|\xi|^2}\,,\quad x\in\mathbb{R}^2
 \end{equation}
 is the Green function of the Laplacian on $\mathbb{R}^2$, thus satisfying $(-\Delta+\omega)\mathcal{G}_{\omega}=\delta$ as a distribution identity, and
 \begin{equation}\label{eq.2d1}
   \beta_{\alpha} (\omega)\;:=\; \alpha + \frac{\gamma}{2\pi} + \frac{1}{2\pi} \ln \left( \frac{\sqrt{\omega}}{2} \right),
\end{equation}
$\gamma\approx 0.577$ denoting the Euler-Mascheroni constant. Observe that $\mathcal{G}_\omega$ displays the local leading singularity $\sim\!(2\pi)^{-1}\log|x|$ as $x\to 0$. Equivalently, $-\Delta_\alpha$ is the following rank-one perturbation, in the resolvent sense, of the ordinary self-adjoint two-dimensional negative Laplacian,
\begin{equation}\label{eq:res_formula}
(-\Delta_{\alpha} +\omega\mathbbm{1})^{-1} g\;=\;(-\Delta+\omega\mathbbm{1})^{-1}g + \frac{\,\langle g, \mathcal{G}_{\omega} \rangle_{L^2}}{\beta_{\alpha}(\omega) }\,\mathcal{G}_{\omega} \,,
\end{equation}
valid for every $g\in L^2(\mathbb{R}^2)$ and $\omega\in\mathbb{R}^+\setminus\{-e_\alpha\}$, with
\begin{equation}
 e_\alpha\;=\;-4 \,e^{-2(2\pi\alpha+\gamma)}\,.
\end{equation}
 In fact, $-\Delta_\alpha$ has essential spectrum $[0,+\infty)$, which is entirely absolutely continuous, and one isolated, non-degenerate, negative eigenvalue $e_\alpha$ with (non-normalised) eigenfunction $\mathcal{G}_{-e_\alpha}$.

 One can also see that on those functions $g\in\mathcal{D}(-\Delta_\alpha)$ vanishing on an open $\mathcal{U}\subset\mathbb{R}^2$, one has $(-\Delta_\alpha g)\big|_{\mathcal{U}}\equiv 0$, and therefore away from the origin $-\Delta_\alpha$ acts precisely as $-\Delta$. The non-trivial interaction with the origin is encoded by a boundary condition of the form \eqref{eq:op_dom}, or equivalent ones: actually, $-\Delta_\alpha$ is only a perturbation of $-\Delta$ in the $L^2$-sector of spherically symmetric functions, and with $s$-wave scattering length equal to $(-2\pi\alpha)^{-1}$. The above stringent interpretation of the contact, point-like nature of the interaction is confirmed by the fact that $-\Delta_\alpha$ can be also reconstructed as the limit, in the resolvent sense, of ordinary Schr\"{o}dinger operators $-\Delta+V_\varepsilon(x)$, with potentials $V_\varepsilon$ on a spatial scale $\varepsilon^{-1}$ that shrink and squeeze to a peaked profile around $x=0$ as $\varepsilon\downarrow 0$ \cite{AHKH-1987-2D,MS-2018-shrinking-and-fractional}.

 The expressions \eqref{eq:op_dom} or \eqref{eq:res_formula} make sense also when one takes formally $\alpha=\infty$, and reproduce in this case the ordinary self-adjoint Laplacian $-\Delta_{\alpha=\infty}=-\Delta$ on $L^2(\mathbb{R}^2)$ with domain $H^2(\mathbb{R}^3)$. The collection $\{-\Delta_\alpha|\alpha\in\mathbb{R}\cup\{\infty\}\}$ provides the whole family of self-adjoint extensions of $-\Delta|_{C^\infty_0(\mathbb{R}^2\setminus\{0\})}$, of which the one with $\alpha=\infty$ is the Friedrichs extension.

 Two natural `regularity' spaces adapted to $-\Delta_\alpha$ are respectively, its operator domain $H^2_{\alpha}(\mathbb{R}^2):=\mathcal{D}(-\Delta_\alpha)$ and form domain $H^1_\alpha(\mathbb{R}^2):= \mathcal{D}[-\Delta_\alpha]$, referred to also as \emph{energy space}. As $-\Delta_\alpha$ is lower semi-bounded, and the bottom of its spectrum is $e_\alpha$, $H^1_\alpha(\mathbb{R}^2)$ is complete with respect to the energy norm
  \begin{equation}\label{eq:energynormH1alpha}
  \|g\|_{H_{\alpha}^1}\;:=\;\big\|(-\Delta_{\alpha}-e_{\alpha}+1)^{\frac{1}{2}}g\big\|_{L^2}
 \end{equation}
 (and a Hilbert space with respect to the scalar product induced by \eqref{eq:energynormH1alpha}), $H^2_{\alpha}(\mathbb{R}^2)$ is dense in $H^1_{\alpha}(\mathbb{R}^2)$, and
 \begin{equation}\label{eq:energynormH1alpha-bis}
  \|g\|_{H_{\alpha}^1}^2\;=\;(-\Delta_{\alpha})[g]+(1-e_{\alpha})\|g\|_{L^2}^2\,,
 \end{equation}
 with the customary notation $(-\Delta_{\alpha})[f]:=\mathcal{Q}_{\alpha}(f,f)$, where $\mathcal{Q}_{\alpha}:H^1_{\alpha}(\mathbb{R}^2)\times H^1_{\alpha}(\mathbb{R}^2)\to\C$ is the closure of the Hermitian form $(f,g)\mapsto\langle -\Delta_\alpha f,g\rangle_{L^2}$. By standard structural properties of the form domain of lower semi-bounded self-adjoint extensions of symmetric operators (see, e.g., \cite{GMO-KVB2017}), one has
 \begin{equation}\label{eq:H1alphastructure}
  H^1_{\alpha}(\mathbb{R}^2)\;=\;H^1(\mathbb{R}^2)\dotplus \mathrm{span}\{ \mathcal{G}_\omega\}\,,\qquad \omega>0\,,
 \end{equation}
 the above sum not depending on the choice of $\omega$ (whereas the decomposition of an element $g\in H^1_{\alpha}(\mathbb{R}^2)$ does), and
 \begin{equation}\label{eq:equiv-norm}
  \|f+c\,\mathcal{G}_\omega \|_{H_{\alpha}^1}^2\;\approx_\omega\;\|f\|^2_{H^1}+|c|^2\,,\qquad f\in H^1(\mathbb{R}^2)\,,\; c\in\mathbb{C},
 \end{equation}
 in the sense of equivalence of norms.

 We shall require a somewhat restrictive condition on the convolutive potential in the Hartree non-linearity, namely that it has a definite sign (positive, for concreteness) and monotonicity, thus re-writing \eqref{singular_hartree} as
 \begin{equation}\label{singular_hartree-theta}
\ii\partial_t\psi=-\Delta_{\alpha}\psi+\theta(w*|\psi|^2)\psi,
\end{equation}
  where the parameter $\theta=\pm 1$ selects, respectively, the defocusing or focusing behaviour of the non-linearity, and
  \begin{equation}\label{ass_w}
  \begin{array}{c}
   \textrm{\emph{$w\in L^{p_1}(\mathbb{R}^2)+L^{p_2}(\mathbb{R}^2)$, with $p_1,p_2\in[1,\infty)$,}} \\
   \textrm{\emph{is non-negative, radial, non-increasing, and not identically zero.}}
  \end{array}
  \end{equation}
 This includes the meaningful cases of pure-power behaviour $w(x)=|x|^{-\eta}$, $\eta\in(0,2)$, and more severe locally-$L^1$ singularities as well. In fact, an amount of intermediate results of the present work are obtained under more general conditions than \eqref{ass_w} (this is the case for the analysis of local well-posedness, which, as generally known and as will emerge in Section \ref{sec:LWP}, or also in the above-mentioned recent study \cite{Adami-Boni-Carlone-Tentarelli-2021}, is insensitive of the sign definiteness or monotonicity of $w$); yet, \eqref{ass_w} is crucial for the technique we employ for the detailed study of the ground state, as further remarked in a moment.

 Our two main focuses in this work are the characterisation of \emph{standing waves} for \eqref{singular_hartree-theta} in the focusing case, and the \emph{global well-posedness} both in the defocusing and focusing case, in the latter scenario covering both the mass sub-critical and the mass critical regimes.

 With the above terminology we mean the following. For $w\in L^{p_1}(\mathbb{R}^2)+L^{p_2}(\mathbb{R}^2)$, with $p_1,p_2\in[1,\infty)$, and $p:=\min\{p_1,p_2\}$, we can prove (Lemma \ref{lem:Hnonlin-controlled})
 %the Gagliardo-Nirenberg type inequality
 \begin{equation}\label{eq:gani-pre}
\int_{\mathbb{R}^2}(w\ast|\psi|^2)|\psi|^2 \,\ud x \;\leqslant C\,\|\psi\|_{H_{\alpha}^1}^{\frac{2}{p}} \|\psi\|_{L^2}^{4-\frac{2}{p}}\qquad\forall\psi\in H_{\alpha}^1\,,
\end{equation}
and we denote by $C_{\textsc{gn}}$ the (Gagliardo-Nirenberg type) optimal constant (dependent on $\alpha$ and $w$) in \eqref{eq:gani-pre}. If $p >1$, then the r.h.s.~is sub-quadratic in $\|\psi\|_{H_{\alpha}^1}$ and the non-linearity is said \emph{mass sub-critical}. If $p=1$,  then the non-linearity is \emph{mass critical}.

 In fact, again when $w\in L^{p_1}(\mathbb{R}^2)+L^{p_2}(\mathbb{R}^2)$, with $p_1,p_2\in[1,\infty)$, we shall prove that the Hartree non-linearity is \emph{energy sub-critical}, explicitly,
 \begin{equation}
  \big\|(w\ast|\psi|^2)\psi\big\|_{H_{\alpha}^{-1}}\;\lesssim\;\|\psi\|_{H_{\alpha}^1}^3
 \end{equation}
 (see estimate \eqref{eq:esc} in the following), where $H^{-1}_\alpha(\mathbb{R}^2)$ denotes the topological dual of $H^1_\alpha(\mathbb{R}^2)$. Actually, in the mass sub-critical case, the Hartree non-linearity is shown to be even locally Lipschitz in the energy space (see estimate \eqref{llhn}). It is therefore natural to study the solution theory for \eqref{singular_hartree-theta} in the energy space $H_{\alpha}^1(\mathbb{R}^2)$.

 Concerning our first focus, we recall that standing wave solutions to \eqref{singular_hartree-theta} when $\theta=-1$ are of the form $e^{-\ii \lambda t}Q_{\alpha}$ for some $\lambda\in\mathbb{R}$ and some non-identically-zero function $Q_{\alpha}\in H^1_{\alpha}(\mathbb{R}^2)$ satisfying
 \begin{equation}\label{t:20}
-\Delta_{\alpha}  Q_{\alpha} + \lambda Q_{\alpha} - (w*|Q_{\alpha}|^2) Q_{\alpha}\;=\;0\,.
\end{equation}
  Observe that \eqref{t:20} does make sense, when $Q_{\alpha}\in H_{\alpha}^1$, at least as an identity in $H^{-1}_\alpha(\mathbb{R}^2)=H^1_{\alpha}(\R^2)^*$.

In order to prove the \emph{existence} of standing waves, we introduce the Weinstein type functional
\begin{equation}\label{eq.wp1}
    %\mathcal{W}^{(\lambda)}(v)
    \mathcal{W}^{(w)}_{\alpha,\lambda}(v)\;:=\; \frac{(-\Delta_{\alpha})[v] +\lambda \|v\|_{L^2}^2}{\;\displaystyle\left(\int_{\mathbb{R}^2}(w\ast|v|^2)|v|^2 \ud x\right)^{\!\frac{1}{2}}}\,,\qquad  v\in H^1_\alpha(\mathbb{R}^2)\setminus\{0\}\,.
\end{equation}
For $H^2_\alpha$-functions the numerator above reads simply $\langle(-\Delta_{\alpha} +\lambda)v,v \rangle_{L^2}$. When $\lambda>|e_{\alpha}|$ the functional \eqref{eq.wp1} is actually bounded from below by a positive constant, as a consequence of estimate \eqref{eq:gani-pre}. Associated to the Weinstein functional we consider the optimisation problem
 \begin{equation}\label{200}
  \inf_{  \substack{ v \in H^1_{\alpha} (\mathbb{R}^2) \\ v\neq 0}  } \mathcal{W}^{(w)}_{\alpha,\lambda}(v)\,.
 \end{equation}

We will show in Lemma \ref{le:EL} that $\mathcal{W}^{(w)}_{\alpha,\lambda}$ is of class $\mathcal{C}^1$ on $H_{\alpha}^1(\R^2)\setminus\{0\}$, and that any ground state $v_{\alpha}$ (i.e., a minimiser to \eqref{200}) satisfies the Euler-Lagrange equation
%$\big(D_v\mathcal{W}^{(w)}_{\alpha,\lambda}\big)(v_{\alpha})=0$, namely
 \begin{equation}\label{t:221}
	-\Delta_{\alpha} v_{\alpha} + \lambda v_{\alpha} - \Lambda_{\alpha,\lambda}^{(w)}(v_{\alpha})(w*|v_{\alpha}|^2) v_{\alpha}\;=\;0
\end{equation}
as an identity in $H_{\alpha}^{-1}(\R^2)$, where
\begin{equation}\label{eq:Lambdav}
	\Lambda_{\alpha,\lambda}^{(w)}(v)\;:=\;\frac{(-\Delta_{\alpha})[v] +\lambda \|v\|_{L^2}^2}{\;\displaystyle\int_{\mathbb{R}^2}(w\ast|v|^2)|v|^2 \ud x}\,,\qquad v\in H_{\alpha}^1(\R^2)\setminus\{0\}.
	\end{equation}
Observe that $\mathcal{W}^{(w)}_{\alpha,\lambda}$ and $\Lambda_{\alpha,\lambda}^{(w)}$ are homogeneous, respectively, of order zero and $-2$ (i.e., $\mathcal{W}^{(w)}_{\alpha,\lambda}(\mu v)=\mathcal{W}^{(w)}_{\alpha,\lambda}(v)$ and $\Lambda_{\alpha,\lambda}^{(w)}(\mu v)=\mu^{-2}\Lambda_{\alpha,\lambda}^{(w)}(v)$ $\forall\mu\in\R\setminus\{0\}$), and moreover $\Lambda_{\alpha,\lambda}^{(w)}(v)>0,$ $\lambda>|e_{\alpha}|$. Therefore, if $v_{\alpha}$ is a minimiser of \eqref{200}, so too is $Q_{\alpha}:=\sqrt{\Lambda_{\alpha,\lambda}^{(w)}(v_{\alpha})}\,v_{\alpha}$, and for the latter function one has $\Lambda_{\alpha,\lambda}^{(w)}(Q_{\alpha})=1$. This explains that from the solutions to the Euler-Lagrange equation \eqref{t:221} the
standing waves for \eqref{singular_hartree-theta} have precisely the form $e^{-i\lambda t}Q_{\alpha}$ with $Q_{\alpha}$ solving \eqref{t:20}.

We have the following result on the existence and symmetry of ground states.

\begin{theorem}\label{vs:10}
	Let $\alpha\in\mathbb{R}$ and $\lambda>|e_{\alpha}|$, and let $w$ satisfy \eqref{ass_w}. Then,
	\begin{equation}\label{2000}
		\widetilde{\mathcal{W}}^{(w)}_{\alpha,\lambda}  \;:=\; \inf_{  \substack{ v \in H^1_{\alpha} (\mathbb{R}^2) \\ v\neq 0}  } \mathcal{W}^{(w)}_{\alpha,\lambda}(v)\;>\;0\,,
	\end{equation}
	and the optimisation problem \eqref{200} admits at least one minimiser. For every minimiser $v_{\alpha}$ to \eqref{200} the following facts hold:
	\begin{itemize}
		\item $v_{\alpha}$ belongs to $H^2_{\alpha}(\mathbb{R}^2)$, hence it admits  the (canonical) representation
		\begin{equation}\label{eq:canore}
			v_{\alpha}  = f_{\alpha}  + \frac{f_{\alpha} (0)}{\beta_{\alpha} (\lambda)} \mathcal{G}_\lambda
		\end{equation}
	with $f_{\alpha}\in H^2(\R^2)$;
		\item $|f_{\alpha}|$ is spherical symmetric, strictly positive, and strictly radially decreasing;
		\item $v_{\alpha}$ has constant phase, that is,
		$$v_{\alpha}=e^{\ii\theta_{\alpha}}\Big(|f_{\alpha}|  +\frac{|f_{\alpha} (0)|}{\beta_{\alpha} (\lambda)} \mathcal{G}_\lambda\Big)$$ for some $\theta_{\alpha}\in[0,2\pi)$.
\end{itemize}
\end{theorem}

Theorem \eqref{vs:10} implies that every minimiser $v_{\alpha}$ for \eqref{200} has a non-trivial singular (i.e., proportional to $\mathcal{G_{\lambda}}$) component, as immediately follows from the canonical representation \eqref{eq:canore} and the fact that $|f_{\alpha}|>0$. In particular, the ground states for $\alpha\in\R$ are different from the $H^2$-ground states of the classical unperturbed Hartree equation.

\begin{remark}
All the properties of the ground states $v_{\alpha}$ provided by Theorem \ref{vs:10} are transferred to the corresponding standing waves profiles $Q_{\alpha}:=\sqrt{\Lambda_{\alpha,\lambda}^{(w)}(v_{\alpha})}\,v_{\alpha}$. In particular, $Q_{\alpha}$ belongs to $H_{\alpha}^2(\mathbb{R}^2)$, it has a non-trivial singular component, and in view of \eqref{eq:op_dom}, \eqref{eq:opaction} and \eqref{t:20} its regular component $F_{\alpha}\in H^2(\R^2)$ satisfies
\begin{equation}\label{eq.imr3}
		-\Delta F_{\alpha}  + F_{\alpha}  = \left(w*\Big(F_{\alpha}  + \frac{F_{\alpha} (0)}{\beta_{\alpha} (\lambda)} \mathcal{G}_\lambda\Big)^2\right) \Big(F_{\alpha}  + \frac{F_{\alpha} (0)}{\beta_{\alpha} (\lambda)} \mathcal{G}_\lambda\Big)\,.
	\end{equation}
\end{remark}

The main difficulty in order to establish the existence of ground states is to choose a minimising sequence for \eqref{200} satisfying suitable symmetry properties which survive in the limit. In the classical case $\alpha=\infty$, thus with the ordinary, unperturbed Laplacian in \eqref{t:20}, it is the Schwartz symmetrisation that provides a final minimising sequence in $H^1(\mathbb{R}^2)$ consisting of spherically symmetric, positive, non-increasing functions. In fact (see, e.g., \cite{Kesavan-symm-2006}), the Schwartz symmetrisation $f\mapsto f^*$ preserves any $L^p$-norm,
   obeys the P\'{o}lya-Szeg\H{o} inequality
   \begin{equation}\label{eq:PSz}
    \|\nabla f^* \|_{L^2} \;\leqslant\; \|\nabla f \|_{L^2}\,,
   \end{equation}
 as well as the Riesz inequality
  \begin{equation}\label{eq.Rin1}
   \int_{\mathbb{R}^2}(w\ast|f|^2)|f|^2\,\ud x \;\leqslant\; \int_{\mathbb{R}^2}(w\ast|f^*|^2)|f^*|^2 \,\ud x\,,
   \end{equation}
  and precisely \eqref{eq.Rin1} is crucial in the construction of a minimising sequence for
  \begin{equation*}
\widetilde{\mathcal{W}}^{(w)}_{\alpha=\infty,\lambda}  \;=\; \inf_{  \substack{ v \in H^1 (\mathbb{R}^2) \\ v\neq 0}  } \mathcal{W}^{(w)}_{\alpha=\infty,\lambda}(v)\;=\;\inf_{  \substack{ v \in H^1 (\mathbb{R}^2) \\ v\neq 0}  }\frac{\|\nabla v\|_{L^2}^2 +\lambda \|v\|_{L^2}^2}{\;\displaystyle\left(\int_{\mathbb{R}^2}(w\ast|v|^2)|v|^2 \ud x\right)^{\!\frac{1}{2}}}
\end{equation*}
  with positive, radial and decreasing $H^1$-functions. At finite $\alpha$, instead, thus with the point-like singular perturbed Laplacian in \eqref{t:20}-\eqref{200}, the elements of the adapted Sobolev space $H^1_\alpha(\mathbb{R}^2)$ have the structure $f + c\,\mathcal{G}_1$ for $f\in H^1(\mathbb{R}^2)$ and $c \in \mathbb{C}$ (see \eqref{eq:H1alphastructure} above), and a generalisation of the Riesz inequality is needed of the form
  \begin{equation}\label{eq.newRiesz}
  \int_{\mathbb{R}^2}(w\ast|f+g|^2)|f+g|^2\, \ud x \;\leqslant\; \int_{\mathbb{R}^2}((w^*)\ast|f^*+g^*|^2)|f^*+g^*|^2\, \ud x
\end{equation}
 in order to perform the minimisation argument. We shall establish \eqref{eq.newRiesz} (Lemma \ref{lbl12}) based on a Brascamp-Friedberg-Lieb-Luttinger inequality.

In addition, an accurate analysis of the equality case in the P\'{o}lya-Szeg\H{o} inequality \eqref{eq:PSz}, together with the $H_{\alpha}^2$-regularity for $v_{\alpha}$ guaranteed by the Euler-Lagrange equation \eqref{t:221}, allows to show that \emph{every} ground state is (up to a constant phase) spherically symmetric, strictly positive, and strictly decreasing (Proposition \ref{pr:machi}), thereby concluding the proof of Theorem \ref{vs:10}.

 Our second main focus concerns the local and global well-posedness of \eqref{singular_hartree-theta} in the energy space. This field is under a comprehensive and well-established control in the classical, unperturbed case $\alpha=\infty$, including also when \eqref{singular_hartree-theta} contains much more singular and non-symmetric convolution potentials, together with electric and magnetic potentials, possibly depending on time (see, e.g., \cite{cazenave,M-2015-nonStrichartzHartree,AMS-2017-globalFinErgNLS} and the references therein).

 We shall associate to \eqref{singular_hartree-theta}, as customary, the mass and the energy of a solution at time $t\geqslant 0$, formally defined, respectively, as
 \begin{eqnarray}
 	\mathcal{M}(t)&:=&\|\psi(t,\cdot)\|_{L^2}^2\,, \label{def:mass} \\
 	\mathcal{E}(t)&:=&\frac{1}{2}(-\Delta_{\alpha})[\psi(t,\cdot)]+\frac{\theta}{4}\int_{\mathbb{R}^2}(w\ast|\psi(t,\cdot)|^2)(x)|\psi(t,x)|^2 \ud x\,. \label{def:energy}
 \end{eqnarray}
They are both formally conserved in time.

\begin{theorem}\label{th:main_ex}
	Let $\alpha\in\mathbb{R}$, $\theta=\pm 1$, and let $w$ satisfy \eqref{ass_w}. Then, for every $\psi_0\in H^1_{\alpha}(\mathbb{R}^2)$, there exist a maximal time $T_{\mathrm{max}}\equiv T_{\mathrm{max}}(\psi_0)\in(0,+\infty]$ and a unique (maximal) solution $\psi\in\mathcal{C}([0,T_{\mathrm{max}}),H^1_{\alpha}(\mathbb{R}^2))$ to the initial value problem
	\begin{equation}\label{cp:shmain}
	\begin{cases}
		\:\ii\partial_t\psi\,=\,-\Delta_{\alpha}\psi +\theta(w\ast|\psi|^2)\psi\,,\\
		\:\psi(0,\cdot)\,=\,\psi_0\,.
	\end{cases}
\end{equation}
 Moreover:
\begin{itemize}
	\item[(i)] (blow-up alternative) if $T_{\mathrm{max}}<+\infty$, then
	$$\lim_{t\uparrow T_{\mathrm{max}}}\|\psi(t)\|_{H_{\alpha}^1}\,=\,+\infty\,;$$
	\item[(ii)] mass $\mathcal{M}(t)$ and energy $\mathcal{E}(t)$ of $\psi$, as defined in \eqref{def:mass}-\eqref{def:energy}, are conserved, i.e, $\mathcal{M}(t)=\mathcal{M}(0)$ and $\mathcal{E}(t)=\mathcal{E}(0)$ $\forall\,t\in(0,T_{\mathrm{max}})$\,;
	\item[(iii)] (continuous dependence on initial data) if, for a sequence $(\psi_0^{(n)})_n$ in $H_{\alpha}^1(\mathbb{R}^2)$, one has $\psi_0^{(n)}\to \psi_0$ in $H_{\alpha}^1(\mathbb{R}^2)$, and if $T\in(0,T_{\mathrm{max}}(\psi_0))$, then, eventually in $n$, the maximal solution $\psi^{(n)}$ to \eqref{singular_hartree-theta} with initial datum $\psi_0^{(n)}$ is defined on $[0,T]$ and satisfies $\psi^{(n)}\to\psi$ in $\mathcal{C}([0,T],H_{\alpha}^1(\mathbb{R}^2))$.
\end{itemize}
 The above solution $\psi$ to \eqref{cp:shmain} is global-in-time (i.e., $T_{\mathrm{max}}=+\infty$) in the following cases:
\begin{itemize}
\item $\theta=1$ (defocusing case);
\item $\theta=-1$ and $p>1$ (focusing, mass sub-critical case);
\item $\theta=-1$ and $p=1$, provided that $\|\psi_0\|_{L^2} < \kappa$ (focusing and mass critical case, with small initial data), for a constant $\kappa$ only depending on $\alpha$ and $\|w\|_{L^{p_1}+L^{p_2}}$ (explicitly, $\sqrt{2}/\kappa$ is the optimal constant for the Gagliardo-Nirenberg type inequality \eqref{eq:gani} below, specialised for $p=1$).
\end{itemize}
In all such cases, $\displaystyle\sup_{t>0}\|\psi(t,\cdot)\|_{H^1_\alpha}<+\infty$.
\end{theorem}

 The material announced so far is discussed in Sections \ref{sec:3D-delta}-\ref{sec:radial}, concerning the characterisation of ground states and the proof of Theorem \ref{vs:10}, and in Sections \ref{sec:LWP}-\ref{sec:GWP}, concerning the local and global well-posedness of \eqref{cp:shmain} and the proof of Theorem \ref{th:main_ex}.

 Notation-wise, we adopt throughout the convention that the $L^2$-product $\langle\cdot,\cdot\rangle_{L^2}$ is linear in the first entry and anti-linear in the second, we write $\langle x\rangle$ for $\sqrt{1+x^2}$, $x\in\mathbb{R}$, we use the short-hand $L_T^p\mathcal{X}$ and $W_T^{s,p}\mathcal{X}$ for the spaces $L^p([0,T],\mathcal{X})$ and $W^{s,p}(I,\mathcal{X})$ for some Banach space $\mathcal{X}$ and $T>0$, and we write $H^s(\mathbb{R}^d)$ for the Sobolev space $W^{s,2}(\mathbb{R}^d)$. By $A\lesssim B$, with $A,B>0$, we mean as customary that $A\leqslant CB$ for some constant $C$, and $A\lesssim_\kappa B$ indicates that $C$ depends on the parameter $\kappa$. $A\approx B$ stands for $A\lesssim B$ and $B\lesssim A$, and $A\approx_\kappa B$ has an obvious analogous meaning. The dual exponent $p/(p-1)$ of the index $p\in[1,+\infty]$ is denoted as usual by $p'$. In identities $f=g$ between measurable functions the `for almost every $x$' declaration is tacitly understood. The convention for the Fourier transform shall be $\widehat{\varphi}(\xi)=(2\pi)^{-\frac{d}{2}}\int_{\mathbb{R}^d}e^{-\ii \xi x}\varphi(x)\,\ud x$\,.

\section{Interpolation and dispersive properties of $-\Delta_\alpha$}\label{sec:3D-delta}

  It is beneficial to exploit certain interpolation and dispersive properties of the two-dimensional operator $-\Delta_\alpha$ in order to single out relevant facts for the forthcoming analysis: the embedding of the energy space into $L^p$-spaces (Lemma \ref{lem:H1alpha-embedding}), the characterisation of the `singular perturbed' Sobolev space $H^s_\alpha(\mathbb{R}^2)$ for intermediate $s\in(0,1)$ (Proposition \ref{main_fd}), a Gagliardo-Nirenberg type inequality adapted to $H^1_\alpha(\mathbb{R}^2)$ (Corollary \ref{cor:GN-ineq}), and dispersive and Strichartz estimates adapted to $-\Delta_\alpha$ (Proposition \ref{pr:did} and Corollary \ref{cor:disp-strich}).

% In  this section we plan to  recall the definition and the main properties of the singular-perturbed Laplace operator $-\Delta_{\alpha} $, following \cite{ABD95}, \cite[Chapter I.5]{albeverio-solvable} and \cite[Section 3]{MO-2016} and in particular Gagliardo - Nirenberg inequality and Strichartz type estimates for $-\Delta_{\alpha} $.

%  Here and below  $\mathcal{G}_{\omega}$ for any $\omega>0$  is the Green function of the Laplacian on $\mathbb{R}^2$, defined by the distributional relation $(-\Delta+\omega)\mathcal{G}_{\omega}=\delta$, that is
% $$\mathcal{G}_{\omega}(x) = (2\pi)^{-2} \int_{\mathbb{R}^2} e^{-\mathrm{i} x\xi} \frac{d\xi}{\omega+|\xi|^2}.$$

 We start with the Green function $\mathcal{G}_{\omega}$ (see \eqref{eq:Gomega-intro} above). From the two-dimensional distribution identity $(\omega-\Delta)\mathcal{G}_{\omega}=\delta$ one easily finds
 \begin{equation}\label{eq.i1}
   \mathcal{G}_{\omega}(x) \;=\;   \frac{\ii}4 H_{0}^{(1)}(\ii |x|\sqrt{\omega}) \;=\; \frac{1}{2\pi}\, K_{0}(|x|\sqrt{\omega})\qquad (\omega>0)\,,
\end{equation}
where $H_0^{(1)}$ and $K_0$ are, respectively, the Hankel function and the modified Bessel function of zero-th order \cite[Chapter 9]{Abramowitz-Stegun-1964}. $\mathcal{G}_{\omega}$ is smooth, spherically symmetric, and also strictly monotone decreasing with $|x|$. Standard asymptotic expansions \cite[Eq.~(9.6.10), (9.6.13), and (9.7.2)]{Abramowitz-Stegun-1964} yield
\begin{equation}\label{eq.ppr1m}
   \mathcal{G}_{\omega}(x)\stackrel{|x|\to 0}{=}  \frac{1}{2\pi} \Big( -\log\Big(\frac{|x|\sqrt{\omega}}{2}\Big) -\gamma  \Big)  + o(1)\,,
\end{equation}
\begin{equation}\label{eq:expdecay}
   \mathcal{G}_{\omega}(x)\stackrel{|x|\to +\infty}{=} \frac{1}{\sqrt{8\pi\sqrt{w}\,|x|\,}}\,e^{-|x|\sqrt{\omega}}(1+O(|x|^{-1})\,.
\end{equation}
 By introducing $\mathcal{G}_0$ through $-\Delta\mathcal{G}_0=\delta$, and hence $\mathcal{G}_0(x)=\frac{1}{2\pi}\log|x|$, \eqref{eq.ppr1m} reads
 \begin{equation}
   \mathcal{G}_{\omega}(x)\stackrel{|x|\to 0}{=} \mathcal{G}_{0}(x) - \beta_0(\omega) + o(1)\,.
 \end{equation}

\begin{lemma}\label{lem:gom1}
For any $ \omega >0$ one has
\begin{eqnarray}
 \mathcal{G}_{\omega} \!\!&\in& \!\! L^p(\mathbb{R}^2)\qquad \forall p \in [1,\infty)\,, \label{eq.fsp1} \\
\mathcal{G}_{\omega} \!\!&\in& \!\! H^s(\mathbb{R}^2)\qquad \forall s \in [0,1)\,, \label{eq.fsp2}
\end{eqnarray}
and for any $\omega_0>0$ and $s\in[0,1)$ there exists $C_{\omega_0,s}>0$ such that
%for any $\omega > \omega_0$
\begin{equation}\label{eq.fsp2a}
    \|\mathcal{G}_{\omega}\|_{H^s}\;\leqslant\; C_{\omega_0,s}\, \omega^{\frac{s}{2}-\frac{1}{2}} \qquad \forall\, \omega>\omega_0\,.
\end{equation}
As $s\uparrow 1$, $C_{\omega_0,s}\to +\infty$.
\end{lemma}

\begin{proof}
\eqref{eq.fsp1} follows from the smoothness of $\mathcal{G}_\omega$ and from \eqref{eq.ppr1m}-\eqref{eq:expdecay}, whereas \eqref{eq.fsp2} follows from
\[
 \begin{split}
   \|\mathcal{G}_{\omega}\|_{H^s}^2\;&=\;\int_{\mathbb{R}^2}(1+|\xi|^2)^s |\widehat{\mathcal{G}_{\omega}}(\xi)|^2\,\ud \xi \\
   &=\;\frac{1}{4\pi^2}\int_{\mathbb{R}^2}\frac{(1+|\xi|^2)^s}{(\omega+|\xi|^2)^2}\,\ud \xi\;<\;+\infty\qquad\forall s\in [0,1)\,.
 \end{split}
\]
 An obvious re-scaling in $\|\mathcal{G}_{\omega}\|_{H^s}^2\leqslant\int_{\mathbb{R}^2}\frac{(1+|\xi|^2)^s}{(\omega+|\xi|^2)^2}\,\ud \xi$ then yields \eqref{eq.fsp2a}.
%
%
%
%
%
% is also obvious and follows from
% $$ \|\mathcal{G}_{\omega}\|_{H^s(\mathbb{R}^2)}^2 = c\|(1+|\xi|^2)^{s/2}\widehat{\mathcal{G}_{\omega}}\|_{L^2(\mathbb{R}^2)}^2 = c \int_{R^2} \frac{(1+|\xi|^2)^s d\xi}{(\omega+|\xi|^2)^2} \lesssim 1 \ \ s \in [0,1) , \omega \in (0,1).$$
% For $\omega >1$ we can use a rescaling and find
% $$ \|\mathcal{G}_{\omega}\|_{H^s(\mathbb{R}^2)}^2\leq  c \int_{R^2} \frac{(\omega+|\xi|^2)^s d\xi}{(\omega+|\xi|^2)^2} \lesssim \omega^{-1+s} \ \ s \in [0,1) , \omega >1.$$
% This completes the proof.
\end{proof}

 The energy space $H^1_\alpha(\mathbb{R}^2)$ satisfies useful embedding properties.

 \begin{lemma}\label{lem:H1alpha-embedding}
  For $\alpha\in\mathbb{R}$ one has the continuous embeddings
  \begin{eqnarray}
   H^1(\mathbb{R}^2) \!\! &\hookrightarrow& \!\! H^1_{\alpha} (\mathbb{R}^2)\,, \label{eq.ene2} \\
   H^1_{\alpha} (\mathbb{R}^2) \!\! &\hookrightarrow& \!\! L^p(\mathbb{R}^2) \qquad \forall p\in[2,\infty)\,. \label{eq.sob1}
  \end{eqnarray}
  Moreover, one has the compact embedding
  \begin{equation}\label{eq.cse1}
		H^1_{\alpha,\mathrm{rad}}(\mathbb{R}^2) \subset\subset L^{q}(\mathbb{R}^2)\qquad \forall q\in(2,\infty)\,,
	\end{equation}
  where $H^1_{\alpha,\mathrm{rad}}(\mathbb{R}^2)$ is the subspace of spherically symmetric $H^1_{\alpha}$-functions.
 \end{lemma}

 \begin{proof}
  \eqref{eq.ene2} follows from the canonical decomposition \eqref{eq:H1alphastructure} and the norm equivalence \eqref{eq:equiv-norm}. Moreover, for a generic $g\in H_{\alpha}^1$ and for $\omega>0$ there are $f_\omega\in H^1(\mathbb{R}^2)$ and $c_\omega\in\mathbb{C}$ so that  $g=f_\omega+c_\omega\mathcal{G}_{\omega}\in H_{\alpha}^1(\mathbb{R}^2)$; owing to \eqref{eq.fsp1}, to the Sobolev embedding $H^1(\mathbb{R}^2)\hookrightarrow L^p(\mathbb{R}^2)$, and to the norm equivalence \eqref{eq:equiv-norm}, one then finds
  \[
   \|g\|_{L^p}\;\lesssim_\omega\;\|f_\omega\|_{L^p}+|c_\omega|\;\lesssim\;\|f_\omega\|_{H^1}+|c_\omega|\;\approx_\omega\,\|g\|_{H^1_\alpha}\,,
  \]
 which proves \eqref{eq.sob1}. Concerning \eqref{eq.cse1}, on account of the canonical decomposition \eqref{eq:H1alphastructure} and the norm equivalence \eqref{eq:equiv-norm}, as well as of the radial symmetry of $\mathcal{G}_\omega$, a generic $H^1_{\alpha}$-bounded sequence $(g_n)_{n\in\mathbb{N}}$ in $H^1_{\alpha,\mathrm{rad}}(\mathbb{R}^2)$ has the form (see \eqref{eq:H1alphastructure} above) $g_n=f_n+c_n\mathcal{G}_1$, where $(f_n)_{n\in\mathbb{N}}$ is a bounded sequence in $H^1_{\mathrm{rad}}(\mathbb{R}^2)$ and $(c_n)_{n\in\mathbb{N}}$ is bounded in $\mathbb{C}$. Owing to the Strauss radial lemma \cite[Lemma 1]{Strauss-CMP1977},
 \[
  H^1_{\mathrm{rad}}(\mathbb{R}^2)\;\subset\subset\; L^{q}(\mathbb{R}^2)\qquad \forall q \in (2,\infty)\,.
 \]
 Thus, up to subsequence, $\|f_n-f_\circ\|_{L^q}\to 0$ and $|c_n\to c_\circ|\to 0$ as $n\to\infty$ for some $f_\circ\in L^q(\mathbb{R}^2)$ and $c_\circ\in\mathbb{C}$, and for arbitrary $q\in(2,\infty)$. This, together with the integrability \eqref{eq.fsp1} of $\mathcal{G}_1$ and again the norm equivalence \eqref{eq:equiv-norm}, then imply that $\|g_n-g_\circ\|_{H^1_{\alpha}}\to 0$ for the corresponding extracted subsequence, where $g_\circ:=f_\circ+c_\circ\mathcal{G}_1$. Having already established the continuous embedding $ H^1_{\alpha} (\mathbb{R}^2) \hookrightarrow L^q(\mathbb{R}^2) $, one then concludes \eqref{eq.cse1}.
 \end{proof}

  Next we consider the Sobolev space $H^s_\alpha(\mathbb{R}^2)$ for intermediate $s\in [0,1)$. In fact, a standard application of the spectral theorem allows to define the fractional power $(\omega-\Delta_\alpha)^{\frac{s}{2}}$ for arbitrary $s\in\mathbb{R}$, whenever $\omega>|e_\alpha|$, as well as the associated singular perturbed Sobolev space
    \begin{equation}\label{eq.defss1}
    H^s_{\alpha} (\mathbb{R}^2) \;:=\;\big\{(\omega-\Delta_\alpha)^{-\frac{s}{2}}h\,\big|\,h\in L^2(\mathbb{R}^2)\big\}
    \end{equation}
  with norm
   \begin{equation}\label{eq.deffss2}
   \|g\|_{H^s_{\alpha}}\;:=\; \|(\omega - \Delta_{\alpha} )^{\frac{s}{2}} g\|_{L^2} \;=\; \|h\|_{L^2}\quad (g=(\omega-\Delta_\alpha)^{-\frac{s}{2}}h\in H^s_{\alpha} (\mathbb{R}^2))\,.
 \end{equation}
  The case $s=0$ corresponds to $L^2(\mathbb{R}^2)$, and for positive $s$ the space $H^{-s}_\alpha(\mathbb{R}^2)$ is the topological dual of $H^{s}_\alpha(\mathbb{R}^2)$. Clearly, the explicit representation \eqref{eq.defss1} for a generic element of $H^s_{\alpha} (\mathbb{R}^2)$, as well as the norm \eqref{eq.deffss2}, are $\omega$-dependent, all such norms being equivalent for the considered $\omega>|e_\alpha|$. In \eqref{eq:energynormH1alpha} the standard $H^1_\alpha$-norm was fixed by the choice $\omega=-e_\alpha+1$.
  Again based on the spectral theorem, one recognises $H^s_\alpha(\mathbb{R}^2)$, when $s\in (0,1)$, as an interpolation space between $L^2(\mathbb{R}^2)$ and $H^1_\alpha(\mathbb{R}^2)$.
  With the customary notation $[\cH_1,\cH_2]_\theta$, $\theta\in(0,1)$, for the interpolation space between two Hilbert spaces $\cH_1$ and $\cH_2$ both embedded continuously in a larger Hilbert space \cite{Bergh-Lofsrom_InterpolationSpaces1976}, one thus writes
  \begin{equation}\label{defss3}
    H^s_{\alpha} (\mathbb{R}^2) \;=\; \Big[L^2(\mathbb{R}^2),  H^1_{\alpha} (\mathbb{R}^2)  \Big]_\theta\,, \qquad \theta=s\,.
 \end{equation}
 This also provides the interpolation inequality
 \begin{equation}\label{eq.defss4}
    \|g\|_{H^s_{\alpha}} \;\leqslant\; \|g\|_{H^1_{\alpha}}^s \,\|g\|_{L^2}^{1-s}\qquad s\in(0,1)\,.
 \end{equation}

  What is less evident, instead, is the actual \emph{coincidence} of this singular perturbed Sobolev space and its classical counterpart.

\begin{proposition} \label{main_fd}
 Let $\alpha\in\mathbb{R}$. For any $ s \in [0,1)$ one has
\begin{equation}\label{eq.leqn1}
  H^s(\mathbb{R}^2) \;=\; H^s_{\alpha} (\mathbb{R}^2)\,.
\end{equation}
\end{proposition}

\begin{proof}
From the interpolation characterisation \eqref{defss3} of $H_{\alpha} ^s(\mathbb{R}^2)$ and the embedding \eqref{eq.ene2} we deduce $H^s(\mathbb{R}^2) \hookrightarrow H^s_{\alpha} (\mathbb{R}^2)$.
% \begin{equation}\label{eq.leqn2}
%    H^s(\mathbb{R}^2) \hookrightarrow H^s_{\alpha} (\mathbb{R}^2).
% \end{equation}
So it remains to check the opposite inclusion
\begin{equation*}\tag{*}\label{eq.leqn3}
     H^s_{\alpha} (\mathbb{R}^2) \;\subset \;H^s(\mathbb{R}^2)
\end{equation*}
 in the non-trivial case $s\in(0,1)$.

 It is convenient to recall (see, e.g., \cite{Bergh-Lofsrom_InterpolationSpaces1976}) that the interpolation space $[\cH_1,\cH_2]_\theta$ consists of those elements $h\in\cH_1\dotplus\cH_2$ (as a direct sum of Hilbert spaces) such that $h=F(\theta)$ for some $F\in\mathscr{F}(\mathcal{H}_1,\mathcal{H}_2)$, where the latter is the space of all bounded and continuous $(\cH_1\dotplus\cH_2)$-valued functions $F$ on the complex strip
 \[
  \mathcal{S}\;:=\;\{z\in\mathbb{C}\,|\,0 \leqslant \mathfrak{Re}\, z \leqslant 1\}
 \]
 which are analytic on the \emph{open} strip $\mathring{ \mathcal{S}}$ and furthermore satisfy the property that the two restrictions of $F$ to the boundary of $\mathcal{S}$, namely the functions $t\mapsto F(\ii t)$ and $t\mapsto F(1+\ii t)$, are continuous from the real line to, respectively, $\cH_1$ or $\cH_2$, and vanish at infinity. Observe that $\mathscr{F}(\mathcal{H}_1,\mathcal{H}_2)$ is a Hilbert space with  norm
 \[
\|F\|_{\mathscr{F}}\;:=\;\max \Big\{\sup_{t\in\mathbb{R}} \|F(\ii t)\|_{\mathcal{H}_1}, \sup_{t\in\mathbb{R}} \|F(1+\ii t)\|_{\mathcal{H}_2}\Big\}\,.
 \]

 In the present case (see \eqref{defss3} above), any  $v \in H^s_{\alpha} (\mathbb{R}^2)$, $s\in(0,1)$, has the form $v=F(s)$ for some $\mathscr{F}(L^2(\mathbb{R}^2), H^1_\alpha(\mathbb{R}^2))$. In particular, owing to the canonical decomposition \eqref{eq:H1alphastructure} for $H^1_\alpha(\mathbb{R}^2)$ and to the $H^1_\alpha$-norm expression \eqref{eq:equiv-norm},
 \[
  F(1+\ii t)\;=\;f_\omega(t)+c_\omega(t)\mathcal{G}_{\omega}
 \]
 for suitable
 %$H^1$-valued function $f_\omega(\cdot)$ and $\mathbb{C}$-valued function $c_\omega(\cdot)$ of $\mathbb{R}$ such that
  \[
   \begin{split}
    f_\omega &\in C(\mathbb{R}_t;H^1(\mathbb{R}^2))\qquad\textrm{with}\quad f_\omega(t)\stackrel{|t|\to\infty}{=}o(1)\,, \\
    c_\omega &\in C(\mathbb{R}_t;\mathbb{C})\qquad\qquad\;\;\textrm{with}\quad c_\omega(t)\stackrel{|t|\to\infty}{=}o(1)\,.
   \end{split}
  \]

  In order to show now that actually $v\in H^s(\mathbb{R}^2)$, let us consider two auxiliary boundary value problems for the Laplace equation in the strip $\mathcal{S}$ (recall that a function $\varphi\in C^2(U,\mathbb{C})$ is \emph{analytic} in the open domain $U \subset \mathbb{C}$ if and only if it satisfies the Laplace equation $(\partial_x^2\varphi)(z)+(\partial_y^2\varphi)(z)=0$ for $z\in U$, where $x=\mathfrak{Re}\,z$ and $y=\mathfrak{Im}\,z$).
  First, we consider the unique solution $\varphi$ to
  \[
   \begin{cases}
    \varphi\in\mathscr{F}(L^2(\mathbb{R}^2), H^1(\mathbb{R}^2))\,, \\
    \varphi(\ii t)=F(\ii t)\,,\quad \varphi(1+\ii t)=f_\omega(t)\quad \forall t\in \mathbb{R}
   \end{cases}
  \]
 (recall that both $F(\ii t)$ and $f_\omega(t)$ vanish as $|t|\to \infty$), then we consider the unique solution $\eta$ to
 \[
   \begin{cases}
    \eta\in\mathscr{F}(\mathbb{C})\,, \\
    \eta(\ii t)=0\,,\quad \eta(1+\ii t)=c_\omega(t)\quad \forall t\in \mathbb{R}\,.
   \end{cases}
  \]
  The interpolation identity $[L^2(\mathbb{R}^2),H^1(\mathbb{R}^2)]_s = H^s(\mathbb{R}^2)$ then implies $\varphi(s)\in H^s(\mathbb{R}^2)$ and therefore, setting
  \[
   \Phi(z)\;:=\;\varphi(z)+\eta(z)\mathcal{G}_\omega\,,\qquad z\in\mathcal{S}\,,
  \]
  $\Phi$ satisfies
    \[
   \begin{cases}
    \Phi\in\mathscr{F}(L^2(\mathbb{R}^2), H^1(\mathbb{R}^2))\subset \mathscr{F}(L^2(\mathbb{R}^2), H^1_\alpha(\mathbb{R}^2))\,, \\
    \Phi(\ii t)=\varphi(\ii t)=F(\ii t)\,, \\
    \Phi(1+\ii t)=f_\omega(t) + c_\omega(t)\mathcal{G}_\omega
   \end{cases}
  \]
   (observe that we used the inclusion \eqref{eq.ene2}). The solution to the latter problem is unique and is precisely $F$. From $\Phi\equiv F$ we then deduce, for $s\in(0,1)$,
   \[
    v\;=\;F(s)\;=\;\Phi(s)\;=\;\varphi(s)+\eta(s)\mathcal{G}_\omega\;\in\;H^s(\mathbb{R}^2)\,,
   \]
  having used in the last step both $\varphi(s)\in H^s(\mathbb{R}^2)$ and \eqref{eq.fsp2}. The claim is thus proved.
\end{proof}

 A direct useful consequence of Proposition \ref{main_fd} and the classical compact Sobolev embedding is the following.

\begin{corollary}\label{cpt_emb}
 Let $\alpha\in\mathbb{R}$. Fixed an open, bounded set $\Omega\subset\mathbb{R}^2$ and $p\in[1,\infty)$, the restriction map $f\mapsto f|_{\Omega}$ is compact from $H_{\alpha}^1(\R^2)$ to $L^p(\Omega)$.	
\end{corollary}

 Proposition \ref{main_fd} also allows to obtain a useful inequality of Gagliardo-Nirenberg type.

\begin{corollary}\label{cor:GN-ineq}
 Let $\alpha\in\mathbb{R}$. For any $ g \in H^1_{\alpha} (\mathbb{R}^2)$ and $p\in[2,\infty)$, one has
 \begin{equation}
  \|g\|_{L^p}\;\lesssim\; \|g\|_{H_{\alpha}^1}^{1-\frac{2}{p}}\|g\|_{L^2}^{\frac{2}{p}}\,.
 \end{equation}
\end{corollary}

\begin{proof}
 For the given $p\in[2,\infty)$ set $s\equiv s(p):=1-\frac{2}{p}\in[0,1)$. Using $H_{\alpha}^{s}(\mathbb{R}^2)= H^{s}(\mathbb{R}^2)$ (Proposition \ref{main_fd}) and the Sobolev embedding $H^{s}(\mathbb{R}^2)\hookrightarrow L^p(\mathbb{R}^2)$, we get
 \[
  \|g\|_{L^p}\;\lesssim\; \|g\|_{H^{s}} \;\approx\; \|g\|_{H_{\alpha} ^{s}}\,.
 \]
 The thesis then follows by interpolation, since $H_{\alpha}^{s}(\mathbb{R}^2)=[L^2(\mathbb{R}^2), H_{\alpha}^1(\mathbb{R}^2)]_{s}$.
\end{proof}

 Last, we turn to the dispersive properties of the linear propagator $(e^{\ii t\Delta_{\alpha} })_{t\in\mathbb{R}}$. In \cite{CMY-2018-2Dwaveop,Yajima-2021-Lpbdd-delta-2D} the $L^p$-boundedness was established, for $p\in(1,\infty)$, of the wave operators for the pair $(-\Delta_{\alpha},-\Delta)$ in two dimensions (similar properties have been also investigated in dimension one and three \cite{DAncona-Pierfelice-Teta-2006,Duchene-Marzuola-Weinstein-2010,Iandoli-Scandone-2017,DMSY-2017}). As a direct consequence, this yields the following dispersive and Strichartz estimates.

\begin{proposition}\label{pr:did}
Let $\alpha\in\mathbb{R}$, and let $P^{(\alpha)}_{\mathrm{ac}}:L^2(\mathbb{R}^2)\to L^2(\mathbb{R}^2)$ be the orthogonal projection onto the absolute continuous subspace for $-\Delta_{\alpha}$, i.e., the orthogonal projection onto $\mathrm{span}\{\mathcal{G}_{-e(\alpha)}\}^{\perp}$. Then the propagator $(e^{\ii t\Delta_{\alpha} })_{t\in\mathbb{R}}$ satisfies
\begin{itemize}
\item[(i)] the dispersive estimate
\[
 \big\|e^{\ii t\Delta_{\alpha}}P^{(\alpha)}_{\mathrm{ac}}\psi\big\|_{L^{p'}}\;\lesssim\; t^{-1}\|\psi\|_{L^p}\qquad \forall p\in(1,2]\,,
\]
\item[(ii)] and the Strichartz estimates
 \begin{eqnarray}
  \|e^{\ii t\Delta_{\alpha}}P^{(\alpha)}_{\mathrm{ac}}\psi\|_{L^r(\mathbb{R}_t, L^p(\mathbb{R}_x^2))} \!\!&\lesssim&\!\! \|\psi\|_{L^2}\,, \label{stri-ho} \\
  \bigg\| \int_0^t e^{\ii(t-s)\Delta_{\alpha}}P^{(\alpha)}_{\mathrm{ac}}F(s)\,\ud s\bigg\|_{L^{r_1}(\mathbb{R}_t, L^{p_1}(\mathbb{R}^2_x))} \!\!&\lesssim&\!\!  \|F\|_{L^{r_2'}(\mathbb{R}_t, L^{p_2'}(\mathbb{R}^2_x))}\,, \label{stri-in}
 \end{eqnarray}
 valid for arbitrary admissible Strichartz pairs $(r_1,p_1)$ and $(r_2,p_2)$, that is, $p_j\in[2,\infty)$ and $p_j^{-1}+r_j^{-1}=\frac{1}{2}$, $j\in\{1,2\}$.
\end{itemize}
\end{proposition}

 For the Proposition above, recall that $p'=\frac{p-1}{p}$, that $-\Delta_{\alpha}$, being self-adjoint, has no singular continuous spectrum, and that $-\Delta_{\alpha}$ admits one eigenvalue only, $e(\alpha)<0$, which is non-degenerate and has eigenfunction $\mathcal{G}_{-e(\alpha)}$.

 Observe that $L^1$-$L^\infty$ dispersive estimates \emph{cannot} hold: for, even a smooth initial datum $\psi$ evolves at later times into a function $e^{\ii t\Delta_{\alpha}}\psi$ that exhibits, for almost every non-zero $t$, a non-trivial singular component proportional to $\mathcal{G}_{\omega}\not\in L^{\infty}(\mathbb{R}^2)$.

 Given the explicit structure of the absolutely continuous subspace for $-\Delta_\alpha$, the above inequalities can be generalised, locally in time, without orthogonal projection.

 \begin{corollary}\label{cor:disp-strich}
  Let $\alpha\in\mathbb{R}$ and $T>0$. Then the propagator $(e^{\ii t\Delta_{\alpha} })_{t\in\mathbb{R}}$ satisfies local-in-time Strichartz estimates
   \begin{eqnarray}
  \|e^{\ii t\Delta_{\alpha}}\psi\|_{L^r([0,T], L^p(\mathbb{R}_x^2))} \!\!&\lesssim_T&\!\! \|\psi\|_{L^2}\,, \label{eq.STr3} \\
  \bigg\| \int_0^t e^{\ii(t-s)\Delta_{\alpha}}F(s)\,\ud s\bigg\|_{L^{r_1}([0,T], L^{p_1}(\mathbb{R}^2_x))} \!\!&\lesssim_T &\!\!  \|F\|_{L^{r_2'}([0,T], L^{p_2'}(\mathbb{R}^2_x))} \label{eq.STr4}
 \end{eqnarray}
 for arbitrary admissible Strichartz pairs $(r_1,p_1)$ and $(r_2,p_2)$, that is, $p_j\in[2,\infty)$ and $p_j^{-1}+r_j^{-1}=\frac{1}{2}$, $j\in\{1,2\}$. The $T$-dependent constants in \eqref{eq.STr3}-\eqref{eq.STr4} above are $O(1)$ as $T\downarrow 0$.
 \end{corollary}

 \begin{proof}
  Since, for $\psi\in L^2(\mathbb{R}^2)$,
  \[
   P^{(\alpha)}_{\mathrm{ac}}\psi\;=\;\psi-\frac{\;\langle \psi,\mathcal{G}_{-e_\alpha}\rangle_{L^2}}{\|\mathcal{G}_{-e_\alpha}\|_{L^2}^2}\,\mathcal{G}_{-e_\alpha}\,,
  \]
  one has
  \[
   e^{\ii t\Delta_{\alpha}}\psi\;=\;e^{\ii t\Delta_{\alpha}}P^{(\alpha)}_{\mathrm{ac}}\psi+\frac{\;\langle \psi,\mathcal{G}_{-e_\alpha}\rangle_{L^2}}{\|\mathcal{G}_{-e_\alpha}\|_{L^2}^2}\,e^{\ii t e_\alpha}\mathcal{G}_{-e_\alpha}\,.
  \]
Taking the $L^r_tL^p_x$-norm in the above identity, for $t\in[0,T]$, and using \eqref{eq.fsp1} and \eqref{stri-ho}, yields \eqref{eq.STr3}. A standard application of the $\mathsf{T}\mathsf{T}^*$-argument and the Christ-Kiselev lemma \cite{Christ-Kiselev-2001} then yields \eqref{eq.STr4}.
 \end{proof}

\section{Existence of ground states}\label{sec:ex_grou}
In this Section we prove the existence of a minimiser of the optimisation problem \eqref{200} claimed in Theorem \ref{vs:10}.

Some preparation is in order. To begin with, as anticipated in the Introduction, our approach requires a control of the Hartree non-linearity by performing a symmetric re-arrangement of its terms that be compatible with the internal structure $H^1_\alpha(\mathbb{R}^2)=H^1(\mathbb{R}^2)\dotplus\mathrm{span}\{\mathcal{G}_\omega\}$ of the energy space. As usual, we denote with $\varphi^*$ the Schwartz symmetrisation of a given measurable function $\varphi:\mathbb{R}^2\to\mathbb{R}^+$ \cite{Kesavan-symm-2006}. Observe that $\varphi\mapsto\varphi^*$ is \emph{not} linear.

 To this aim, we can replace the ordinary Riesz inequality \eqref{eq.Rin1} with the announced modification \eqref{eq.newRiesz} by exploiting the following estimate, that is fair to refer to collectively as the Brascamp-Friedberg-Lieb-Luttinger inequality (BFLL for short), as it was conjectured by Friedberg and Luttinger in \cite{Luttinger-Friedberg-1976}, and demonstrated by Brascamp, Lieb, and Luttinger in \cite{Brascamp-Lieb-Luttinger-1974}: given $N,n\in\mathbb{N}$ and the multi-linear functional
 \begin{equation}
  I(F_1,\dots,F_N)\;:=\;\int_{\mathbb{R}^{2n}} \prod_{j=1}^N F_j(L_j(X))\,\ud X
 \end{equation}
 on non-negative measurable functions $F_1,\dots,F_N$ on $\mathbb{R}^2$, where $X\equiv(x_1,\dots,x_N)\in\mathbb{R}^{2N}$ and $L_1,\dots,L_N:\mathbb{R}^{2n}\to\mathbb{R}^2$ are surjective linear maps of the form $ L_j(X) = \sum_{\ell=1}^n a_{j\ell} x_\ell$ for given real numbers $a_{j\ell}$, $j\in\{1,\dots,N\}$, $\ell\in\{1,\dots,n\}$, one has
 \begin{equation}\label{eq.BL6}
   I(F_1, \cdots, F_n) \;\leqslant\; I(F_1^*, \cdots, F_n^*)\,.
\end{equation}
 (Clearly, \eqref{eq.BL6} is non-trivial only for $N>n$.)

\begin{lemma} \label{lbl12}
For positive measurable functions $w,f,g$ on $\mathbb{R}^2$ one has
\begin{equation}\label{eq.Rin3}
  \int_{\mathbb{R}^2}(w\ast(f+g)^2)(f+g)^2 \, \ud x \;\leqslant\; \int_{\mathbb{R}^2}((w^*)\ast(f^*+g^*)^2)(f^*+g^*)^2\,\ud  x\,,
\end{equation}
provided that both sides of the inequality are finite.
\end{lemma}

\begin{proof}[Proof of Lemma \ref{lbl12}]
We have to check the validity of
 \begin{equation*}\tag{i}\label{eq.BL2}
 \begin{aligned}
   & \iint_{\mathbb{R}^2\times\mathbb{R}^2}w(x_1-x_2)(f(x_1)+g(x_1))^2(f(x_2)+g(x_2))^2 \,\ud x_1 \,\ud x_2  \\
   & \qquad\leqslant \iint_{\mathbb{R}^2\times\mathbb{R}^2}w^*(x_1-x_2)(f^*(x_1)+g^*(x_1))^2(f^*(x_2)+g^*(x_2))^2 \,\ud x_1 \,\ud x_2 \,.
      \end{aligned}
\end{equation*}
In turn, \eqref{eq.BL2} follows once one checks
\begin{equation*}\tag{ii}\label{eq.BL3}
\begin{aligned}
   & \iint_{\mathbb{R}^2\times\mathbb{R}^2}w(x_1-x_2)f(x_1)g(x_1)f(x_2)^2  \,\ud x_1 \,\ud x_2
   \\&\qquad \leqslant\iint_{\mathbb{R}^2\times\mathbb{R}^2}w^*(x_1-x_2)f^*(x_1)g^*(x_1)f^*(x_2)^2  \,\ud x_1 \,\ud x_2
      \end{aligned}
\end{equation*}
and
\begin{equation*}\tag{iii}\label{eq.BL4}
\begin{aligned}
   & \iint_{\mathbb{R}^2\times\mathbb{R}^2}w(x_1-x_2)f(x_1)g(x_1)f(x_2)g(x_2)\,\ud x_1 \,\ud x_2 \\
   &\qquad\leqslant \iint_{\mathbb{R}^2\times\mathbb{R}^2}w^*(x_1-x_2)f^*(x_1)g^*(x_1)f^*(x_2)g^*(x_2) \,\ud x_1 \,\ud x_2\,,
      \end{aligned}
\end{equation*}
 in combination with the ordinary Riesz inequality \eqref{eq.Rin1}.
 Now, \eqref{eq.BL3} follows from the BFLL inequality \eqref{eq.BL6} with the special choice $N=5$, $n=2$, and
 \[
  \begin{split}
   & F_1(L_1(X))\,=\,w(x_1-x_2)\,,\quad F_2(L_2(X))\,=\,f(x_1)\,, \\
   & F_3(L_3(X))\,=\,g(x_1)\,,\quad F_4(L_4(X))\,=\,F_5(L_5(X))\,=\,f(x_2)\,.
  \end{split}
 \]
 With an analogous choice one establishes also \eqref{eq.BL4}.
\end{proof}

 Concerning the Hartree non-linearity, we further need for multiple purposes a standard control of its integrability and local Lipschitz property in suitable $L^p$-sense and $H^1_\alpha$-sense.

\begin{lemma}\label{lem:Hnonlin-controlled}
	Let $w\in L^p(\mathbb{R}^2)$ for $p\in[1,\infty)$.
	\begin{itemize}
		\item[(i)] Let $q_1,q_2,q_3,r\in [1,\infty]$ be such that $1-\frac{1}{p}=\frac{1}{q_1}+\frac{1}{q_2}+\frac{1}{q_3}-\frac{1}{r}$. Then, for every $\psi_j\in L^{q_j}(\mathbb{R}^2)$, $j\in\{1,2,3\}$,
		\begin{equation}\label{eq:yoho}
			\|(w*(\psi_1\psi_2))\psi_3\|_{L^{r}}\;\lesssim \;\|w\|_{L^p}\prod_{j=1}^3\|\psi_j\|_{L^{q_j}}\,.
		\end{equation}
		In particular,
		\begin{equation}\label{eq:L2Hartree}
		 \big\|(w*|\psi|^2)\psi\big\|_{L^{2}}\;\lesssim \;\|w\|_{L^p}\|\psi\|^3_{H^1_\alpha}\,.
		\end{equation}
	\item[(ii)] Let $q_1,q_2,r\in[1,\infty]$ be such that $1-\frac{1}{p}=\frac{2}{q_1}+\frac{1}{q_2}-\frac{1}{r}$. Then, for every $\psi_1,\psi_2\in L^{q_1}(\mathbb{R}^2)\cap L^{q_2}(\mathbb{R}^2)$,
		\begin{equation}\label{eq:dig}
		\begin{split}
		 &\big\|(w*|\psi_1|^2)\psi_1-(w*|\psi_2|^2)\psi_2\big\|_{L^r} \\
		 &\qquad \lesssim\; \|w\|_{L^p}\big(\|\psi_1\|_{L^{q_1}}^2+\|\psi_2\|_{L^{q_1}}^2\big)\|\psi_1-\psi_2\|_{L^{q_2}}\,.
		\end{split}
		\end{equation}
	\item[(iii)] Given $\psi_j\in L^{\frac{4p}{2p-1}}(\mathbb{R}^2)$, $j\in\{1,2,3,4\}$,
	\begin{equation}\label{quartic_p}
		\|(w*(\psi_1\psi_2))\psi_3\psi_4\|_{L^1}\;\lesssim \;\|w\|_{L^p}\prod_{j=1}^4\|\psi_j\|_{L^{\frac{4p}{2p-1}}}\,.
	\end{equation}
In particular, given $\psi_j\in H_{\alpha}^1(\mathbb{R}^2)$, $j\in\{1,2,3,4\}$,
\begin{equation}\label{quartic}
	\|(w*(\psi_1\psi_2))\psi_3\psi_4\|_{L^1}\;\lesssim \;\|w\|_{L^p}\prod_{j=1}^4\|\psi_j\|_{H_{\alpha}^1}\,.
\end{equation}
\item[(iv)] Let $\alpha\in\mathbb{R}$ and $\psi\in H^1_\alpha(\mathbb{R}^2)$. Then
 \begin{equation}\label{eq:gani}
 \bigg|\int_{\mathbb{R}^2}(w*|\psi|^2)|\psi|^2\,\ud x\,\bigg|\;\lesssim\; \|w\|_{L^p}\|\psi\|_{H_{\alpha}^1}^{\frac{2}{p}} \|\psi\|_{L^2}^{4-\frac{2}{p}}\,.
\end{equation}
	\end{itemize}
\end{lemma}

\begin{proof}
	\eqref{eq:yoho} is a straightforward consequence of Young and H\"{o}lder inequalities.
	\eqref{eq:L2Hartree} follows from \eqref{eq:yoho} re-written in the form
	\[
	 \|(w*|\psi|^2)\psi\|_{L^{2}}\;\lesssim\;\|w\|_{L^{p}}\|\psi\|_{L^{\frac{6p}{3p-2}}}^3
	\]
and from the embedding \eqref{eq.sob1}.
	Along the same way, combining \eqref{eq:yoho} with the identity
	\begin{equation}\label{pdpd}
	 \begin{split}
	  &(w*|\psi_1|^2)\psi_1-(w*|\psi_2|^2)\psi_2 \\
	  &\qquad=\;(w*|\psi_1|^2)(\psi_1-\psi_2)+\big(w*\big((|\psi_1|+|\psi_2|)(|\psi_1|-|\psi_2|)\big)\big)\psi_2\,,
	 \end{split}
	\end{equation}
 one obtains \eqref{eq:dig}. \eqref{quartic_p} follows from \eqref{eq:yoho} with the choice $q_1=q_2=2q_3$, $r=1$, and with $\psi_3$ replaced by $\psi_3\psi_4$. \eqref{quartic} follows from \eqref{quartic_p} and the embedding \eqref{eq.sob1}. Last, the Sobolev embedding $H^{\frac{1}{2p}}(\mathbb{R}^2)\hookrightarrow L^{\frac{4p}{2p-1}}(\mathbb{R}^2)$ and the interpolation inequality
 \[
  \|\psi\|_{H^{\frac{1}{2p}}} \;\lesssim\; \|\psi\|^{\frac{1}{2p}}_{H^1_{\alpha} }\|\psi\|_{L^2}^{1-\frac{1}{2p}}
 \]
(see \eqref{eq.defss4} above) imply the estimate
\[
 \|\psi\|_{L^{\frac{4p}{2p-1}}}^4\;\lesssim\; \|\psi\|^{\frac{2}{p}}_{H^1_{\alpha} }\|\psi\|_{L^2}^{4-\frac{2}{p}}\,,%\qquad r=\frac{4p}{2p-1}\,,
\]
which, combined with \eqref{quartic_p} (with the choice $\psi_j=\psi$, $j\in\{1,2,3,4\}$) yields \eqref{eq:gani}.
\end{proof}

In the preceding two lemmas, the even symmetry of $w$ was not needed. Instead, when such a condition is assumed (un particular, under \eqref{ass_w}), standard change of variable and Fubini's theorem guarantee the following property that for convenience we single out in a separate lemma.

\begin{lemma}\label{le:change.w}
Let $w,f,g$ be measurable functions on $\R^2$, and assume that $w$ is real-valued and even. Suppose moreover that
$$\iint_{\R^2\times\mathbb{R}^2}\big|\,w(x-y)\,f(x)\,g(y)\,\big|\,\ud x\,\ud y\;<\;+\infty\,.$$
Then
$$\int_{\R^2}(w*f)\,\overline{g}\,\ud x\;=\;\int_{\R^2}f\,\overline{(w*g)}\,\ud x\,.$$
\end{lemma}

  As a further preparation, let us provide a more explicit connection between the equivalent expressions \eqref{eq:energynormH1alpha}-\eqref{eq:energynormH1alpha-bis} and \eqref{eq:equiv-norm} for the $H^1_\alpha$-norm.
  Regarding, as mentioned, $-\Delta_\alpha$ as a self-adjoint extension of a symmetric operator whose Friedrichs extension is the self-adjoint $-\Delta$, and taking into account that the Friedrichs and the $\alpha$-extension differ, in the resolvent sense, by the rank-one projection onto $\mathcal{G}_\omega$ (see \eqref{eq:res_formula}), one has the following Birman formula (see, e.g., \cite[Theorem 7]{GMO-KVB2017}) for the quadratic form of $-\Delta_\alpha$ evaluated on a generic element $g=f+c\,\mathcal{G}_\omega\in\mathcal{D}[-\Delta_\alpha]=H^1_\alpha(\mathbb{R}^2)$, where $f\in H^1(\mathbb{R}^2)$, $c\in\mathbb{C}$, and $\omega>0$:
   \begin{equation}\label{eq:Bformula}
    (-\Delta_\alpha)[g]+\omega\|g\|_{L^2}^2 \;=\;\|\nabla f\|_{L^2}^2+\omega\|f\|_{L^2}^2+\frac{|c|^2}{\beta_\alpha(\omega)}\,.
  \end{equation}
  (Observe that \eqref{eq:Bformula} is consistent with what one deduces directly from \eqref{eq:op_dom}-\eqref{eq:opaction} in the particular case when $g = f + f(0) \mathcal{G}_\omega/\beta_\alpha(\omega) \in\mathcal{D}(\Delta_\alpha)=H^2_\alpha(\mathbb{R}^2)$, where now $f \in H^2(\mathbb{R}^2)$: indeed,
  \[
   \begin{split}
    (-\Delta_\alpha)[g]+\omega\|g\|_{L^2}^2\;&=\;\langle (-\Delta_\alpha + \omega)g, g \rangle_{L^2} \;=\; \langle (-\Delta + \omega) f, g \rangle_{L^2} \\
    &=\;  \langle (-\Delta + \omega) f, f \rangle_{L^2}+ \frac{\overline{f(0)}}{\beta_\alpha(\omega)}\,\langle (-\Delta + \omega) f, \mathcal{G}_\omega \rangle_{L^2} \\
    &=\;\|\nabla f\|_{L^2}^2+\omega\|f\|_{L^2}^2+\frac{|f(0)|^2}{\beta_\alpha(\omega)}\,,
   \end{split}
  \]
  which is precisely \eqref{eq:Bformula}.)
  When $\omega>|e_\alpha|$, one has $\beta_{\alpha}(\omega)>0$  and
   \begin{equation}\label{eq:pre-eq}
    \big\|(-\Delta_\alpha+\omega)^{\frac{1}{2}}g\big\|_{L^2}^2\;=\;(-\Delta_\alpha)[g]+\omega\|g\|_{L^2}^2\;=\;\|\nabla f\|_{L^2}^2+\omega\|f\|_{L^2}^2+\frac{|c|^2}{\beta_\alpha(\omega)}\,.
   \end{equation}
  In particular, we took $\omega=-e_\alpha+1$ in \eqref{eq:energynormH1alpha}-\eqref{eq:energynormH1alpha-bis} to fix the standard expression for the $H^1_\alpha$-norm, and the equivalence of norms \eqref{eq:equiv-norm} is precisely a consequence of \eqref{eq:pre-eq}. We also introduced the Weinstein functional \eqref{eq.wp1} writing the numerator therein as the l.h.s.~of \eqref{eq:Bformula} with $\omega=\lambda>|e_\alpha|$.

  Next, for the minimisation \eqref{200} we show that one can select a minimising sequence with certain explicit features as follows.

  \begin{lemma}\label{lem:minimis-seq-feat}
   Let $\alpha\in\mathbb{R}$ and $\lambda>|e_\alpha|$. Assume that $w$ satisfies the condition \eqref{ass_w}. The minimisation problem \eqref{200} admits a minimising sequence $(v_n)_{n\in\mathbb{N}}$ in $H^1_{\alpha}(\mathbb{R}^2)$ such that, for every $n$,
   \begin{equation}\label{eq:struct-min}
    v_n\;=\;f_n+c_n\mathcal{G}_\lambda
   \end{equation}
   for some $c_n\geqslant 0$ and some non-negative, radial function $f_n\in H^1(\mathbb{R}^2)$ that is monotone decreasing in $|x|$, and additionally satisfies
   \begin{equation}\label{eq.asp1m}
   \int_{\mathbb{R}^2} (w*v_n^2) \,v_n^2 \,\ud x \;=\; 1\,.
\end{equation}
 Such a sequence is uniformly bounded in $H^1_{\alpha}(\mathbb{R}^2)$.
  \end{lemma}

  \begin{proof}
   Any minimising sequence $(\mathfrak{v}_n)_{n\in\mathbb{N}}$ has the structure $\mathfrak{v}_n=\mathfrak{f}_n+\mathfrak{c}_n\mathcal{G}_\lambda$ with $\mathfrak{f}_n\in H^1(\mathbb{R}^2)$ and $\mathfrak{c}_n\in\mathbb{C}$, on account of \eqref{eq:H1alphastructure}. For any such $(\mathfrak{v}_n)_{n\in\mathbb{N}}$, the new sequence $(\widetilde{v}_n)_{n\in\mathbb{N}}$ defined by
   \[
    \widetilde{v}_n \,:=\, \widetilde{f}_n + c_n\mathcal{G}_\lambda\,,\quad\textrm{ with } \widetilde{f}_n\,:=\,|\mathfrak{f}_n|\,\textrm{ and }\, c_n\,:=\,|\mathfrak{c}_n|
   \]
   still belongs to $H^1_\alpha(\mathbb{R}^2)$, because $|\mathfrak{f}_n|\in H^1(\mathbb{R}^2)$
   %, and has by construction non-negative regular part $\widetilde{f}_n$ and non-negative coefficient $c_n$ of the singular part.
   Beside,
      \[
   |\mathfrak{v}_n| \;\leqslant\; \widetilde{v}_n\,,\qquad \int_{\mathbb{R}^2} (w*|\mathfrak{v}_n|^2) \,|\mathfrak{v}_n|^2 \,\ud x \;\leqslant\; \int_{\mathbb{R}^2} (w*\widetilde{v}_n^2) \,\widetilde{v}_n^2 \,\ud x \,,
   \]
 and
 \[
  \begin{split}
   (-\Delta_{\alpha})[\widetilde{v}_n] +\lambda \|\widetilde{v}_n\|_{L^2}^2\;&=\;\|\nabla |\mathfrak{f}_n|\|_{L^2}^2+\lambda\|\mathfrak{f}_n\|_{L^2}^2+\frac{|\mathfrak{c}_n|^2}{\beta_\alpha(\lambda)} \\
   &\leqslant\;\|\nabla \mathfrak{f}_n\|_{L^2}^2+\lambda\|\mathfrak{f}_n\|_{L^2}^2+\frac{|\mathfrak{c}_n|^2}{\beta_\alpha(\lambda)}\;=\; (-\Delta_{\alpha})[\mathfrak{v}_n] +\lambda \|\mathfrak{v}_n\|_{L^2}^2\,,
  \end{split}
 \]
 having used \eqref{eq:Bformula} and $|\nabla |\mathfrak{f}_n||\leqslant|\nabla\mathfrak{f}_n|$. Thus,
 \[
  \mathcal{W}^{(w)}_{\alpha,\lambda}(\widetilde{v}_n) \;\leqslant\; \mathcal{W}^{(w)}_{\alpha,\lambda}(\mathfrak{v}_n)\,,
 \]
 showing that $(\widetilde{v}_n)_{n\in\mathbb{N}}$ too is a minimising sequence. The further sequence $(v_n)_{n\in\mathbb{N}}$ defined by
 \[
  v_n\,:=\,f_n+c_n\mathcal{G}_\lambda\,,\qquad f_n\,:=\,(\widetilde{f}_n)^*
 \]
 (namely, with the sole Schwartz symmetrisation of the regular part of each $\widetilde{v}_n$) still belongs to $H^1_\alpha(\mathbb{R}^2)$, since $(\widetilde{f}_n)^*\in H^1(\mathbb{R}^2)$, owing to the invariance of the $L^2$-norm under symmetrisation and the P\'{o}lya-Szeg\H{o} inequality \eqref{eq:PSz}. Moreover, by construction, the coefficient $c_n$ of the singular part is non-negative, and the regular part $f_n$ is non-negative, radial, and monotone decreasing in $|x|$. Lemma \ref{lbl12}, the P\'{o}lya-Szeg\H{o} inequality again, and the spherical symmetry of $w$, imply
  \[
  \frac{\big\|\nabla (\widetilde{f}_n)^*\big\|_{L^2}^2+\lambda\big\|(\widetilde{f}_n)^*\big\|_{L^2}^2+c_n^2/\beta_\alpha(\lambda)}{\;\displaystyle\left(\int_{\mathbb{R}^2}(w\ast|v_n|^2)|v_n|^2 \ud x\right)^{\!\frac{1}{2}}}\;\leqslant\;\frac{\|\nabla \widetilde{f}_n\|_{L^2}^2+\lambda\|\widetilde{f}_n\|_{L^2}^2+c_n^2/\beta_\alpha(\lambda)}{\;\displaystyle\left(\int_{\mathbb{R}^2}(w\ast|\widetilde{v}_n|^2)|\widetilde{v}_n|^2 \ud x\right)^{\!\frac{1}{2}}}\,,
 \]
  that is,
   \[
  \mathcal{W}^{(w)}_{\alpha,\lambda}(v_n) \;\leqslant\;\mathcal{W}^{(w)}_{\alpha,\lambda}(\widetilde{v}_n)\,.
 \]
  Thus, $(v_n)_{n\in\mathbb{N}}$ too is a minimising sequence. An obvious multiplicative re-sizing of $v_n$ allows to match the additional condition \eqref{eq.asp1m}. We have so far shown that a minimising sequence can be indeed chosen with all the properties stated in the Lemma, but for the uniform $H^1_{\alpha}$-boundedness, which we prove now. Since
  \[
 \begin{split}
    \|v_n\|_{H^1_{\alpha}}^2\;& =  (-\Delta_{\alpha})[v_n]+(1-e_\alpha) \|v_n\|_{L^2}^2 \leq\;\big\|(-\Delta_\alpha+\lambda)^{\frac{1}{2}}v_n\big\|_{L^2}^2 \\
    &=\;\|\nabla f_n\|_{L^2}^2+\lambda\|f_n\|_{L^2}^2+|c_n|^2/\beta_\alpha(\omega)\;=\;\mathcal{W}^{(w)}_{\alpha,\lambda}(v_n)\;\xrightarrow{\,n\to\infty\,}\;\widetilde{\mathcal{W}}^{(w)}_{\alpha,\lambda}\,,
 \end{split}
 \]
 having used \eqref{eq:pre-eq} with $\omega=\lambda \geqslant 1-e_\alpha$ and \eqref{eq.asp1m}, $(v_n)_{n\in\mathbb{N}}$ is indeed bounded in $H^1_\alpha(\mathbb{R}^2)$.
\end{proof}

 For the interest per se and for later application (Proposition \ref{pr:machi}), let us single out the key property emerged from the proof of Lemma \ref{lem:minimis-seq-feat}:

 \begin{corollary}\label{cor:loweringWeinstein}
  Let $\alpha\in\mathbb{R}$ and $\lambda>|e_\alpha|$. Assume that $w$ satisfies the condition \eqref{ass_w}. Given a non-zero element $f+c\,\mathcal{G}_\lambda\in H^1_\alpha(\mathbb{R}^2)$, according to the decomposition \eqref{eq:H1alphastructure}, then also $|f|+|c|\,\mathcal{G}_\lambda$ and $|f|^*+|c|\,\mathcal{G}_\lambda$ belong to $H^1_\alpha(\mathbb{R}^2)\setminus\{0\}$ (where $|f|^*$ is the Schwartz symmetrisation of $|f|$), and
  \begin{equation}\label{eq:loweringWeinstein}
   \mathcal{W}^{(w)}_{\alpha,\lambda}\big(|f|^*+|c|\,\mathcal{G}_\lambda\big) \;\leqslant\; \mathcal{W}^{(w)}_{\alpha,\lambda}\big(|f|+|c|\,\mathcal{G}_\lambda\big)\;\leqslant\;\mathcal{W}^{(w)}_{\alpha,\lambda}\big(f+c\,\mathcal{G}_\lambda\big)\,.
  \end{equation}
 \end{corollary}

For the remaining part of this Section, $(v_n)_{n\in\mathbb{N}}$ shall denote a minimising sequence with the features given by Lemma \ref{lem:minimis-seq-feat}. In particular,
 \begin{equation}
  \mathcal{W}^{(w)}_{\alpha,\lambda}(v_n)\;=\;\|\nabla f_n\|_{L^2}^2+\lambda\|f_n\|_{L^2}^2+\frac{c_n^2}{\beta_\alpha(\lambda)}\,,\quad \lambda>|e_\alpha|\,,\quad\beta_\alpha(\lambda)>0\,.
 \end{equation}
 Because of the uniform $H^1_{\alpha}$-boundedness of $(v_n)_{n\in\mathbb{N}}$, there exists
 \begin{equation}\label{eq:vcirc}
  v_\circ\;=\;f_\circ+c_\circ \mathcal{G}_\lambda\,\in\, H^1_{\alpha}(\mathbb{R}^2)\,,
 \end{equation}
 %$v_\circ=f_\circ+c_\circ \kappa_{\alpha,\lambda}\in H^1_{\alpha}(\mathbb{R}^2)$,
possibly a null function, such that $v_n \rightharpoonup v_\circ$ $H^1_{\alpha}$-weakly, which in turn implies that $f_n \rightharpoonup  f_\circ$ weakly in $H^1(\mathbb{R}^2)$ and $c_n\to c_\circ$. Moreover, $f_\circ$ is necessarily non-negative, spherically symmetric, and radially decreasing, and $c_\circ\geqslant 0$. This, and the properties of $\mathcal{G}_\lambda$, imply that $v_\circ$ too is non-negative, spherically symmetric, and radially decreasing.

 The next ingredient is a \emph{compactness result} for the Hartree type non-linearity.

\begin{lemma}\label{cl1}
 Let $\alpha\in\mathbb{R}$. Assume that $w$ satisfies the condition \eqref{ass_w} and that $(g_n)_{n\in\mathbb{N}}$ is a uniformly bounded sequence in $H^1_{\alpha} (\mathbb{R}^2)$ consisting of non-negative, spherically symmetric functions and such that $g_n \rightharpoonup g_\circ$ $H^1_{\alpha}$-weakly. Then, up to subsequence,
 \begin{equation}\label{eq.lim12}
  \lim_{n\to\infty}\int_{\mathbb{R}^2} (w*g_n^2) \,g_n^2 \,\ud x\;=\;\int_{\mathbb{R}^2} (w*g_\circ^2) \,g_\circ^2 \,\ud x\,.
 \end{equation}
\end{lemma}

\begin{proof} Clearly, $g_\circ$ is non-negative and spherically symmetric too. From
	\[
	 \begin{split}
	  &\bigg|\int_{\mathbb{R}^2} (w*g_n^2) \,g_n^2 \,\ud x-\int_{\mathbb{R}^2} (w*g_\circ^2) \,g_\circ^2 \,\ud x\bigg| \\
	  &\quad\leqslant\;\int_{\mathbb{R}^2} (w*g_n^2) |g_n-g_\circ|(g_n+g_\circ)\,\ud x+\int_{\mathbb{R}^2}\Big(w*\big(|g_n-g_\circ|(g_n+g_\circ)\big)\Big)|g_\circ|^2\,\ud x
	 \end{split}
	\]
       and from \eqref{quartic_p} one finds
       \[
        \begin{split}
         &\bigg|\int_{\mathbb{R}^2} (w*g_n^2) \,g_n^2 \,\ud x-\int_{\mathbb{R}^2} (w*g_\circ^2) \,g_\circ^2 \,\ud x\bigg| \\
         &\qquad\lesssim\;\|w\|_{L^p}\Big( \|g_n\|^3_{L^{\frac{4p}{2p-1}}}+\|g_\circ\|^3_{L^{\frac{4p}{2p-1}}}\Big)\|g_n-g_\circ\|_{L^{\frac{4p}{2p-1}}}\,.
        \end{split}
       \]
     As $\frac{4p}{2p-1}\in(2,4]\subset(2,\infty)$, the compact embedding $H^1_{\alpha,\mathrm{rad}}(\mathbb{R}^2)\hookrightarrow L^{\frac{4p}{2p-1}}(\mathbb{R}^2)$ from Lemma \ref{lem:H1alpha-embedding} is applicable; using this fact, and the uniform $H^1_{\alpha}$-boundedness, one concludes that up to extracting a sub-sequence, the expression above converges to $0$ as $n\to\infty$.
\end{proof}

 \begin{corollary}
  The function \eqref{eq:vcirc} is non-zero.
 \end{corollary}

  \begin{proof}
  The constraint \eqref{eq.asp1m} from Lemma \ref{lem:minimis-seq-feat}, and Lemma \ref{cl1}, imply that, up to re-fining the minimising sequence $(v_n)_{n\in\mathbb{N}}$,
    \begin{equation*}
  \int_{\mathbb{R}^2} (w*v_\circ^2) \,v_\circ^2 \,\ud x\;=\;\lim_{n\to\infty}\int_{\mathbb{R}^2} (w*v_n^2) \,v_n^2 \,\ud x\;=\;1\,,
 \end{equation*}
   which prevents $v_\circ$ from being almost-everywhere vanishing.
  \end{proof}

  We can finally assemble the above preparations and address the core part of the proof of the existence of ground states (Theorem \ref{vs:10}).

  We saw that, as $n\to\infty$,
  \begin{equation}\label{eq:arecap}
  \begin{split}
   &\|\nabla f_n\|_{L^2}^2+\lambda\|f_n\|_{L^2}^2+c_n^2/\beta_\alpha(\lambda)\;\searrow\;\widetilde{\mathcal{W}}^{(w)}_{\alpha,\lambda}\,, \\
   &\qquad f_n \rightharpoonup f_\circ\quad \textrm{$H^1_\alpha$-weakly}\,,\qquad c_n\to c_\circ
  \end{split}
  \end{equation}
for some non-zero $v_\circ=f_\circ+c_\circ \mathcal{G}_\lambda\in H^1_{\alpha,\mathrm{rad}}(\mathbb{R}^2)$ ($\lambda>|e_\alpha|$, $\beta_{\alpha}(\lambda)>0$) with $f_\circ$ being non-negative, spherically symmetric, and radially decreasing. Moreover $v_\circ $ satisfies the constraint
  \begin{equation}\label{eq:vcircconstraint}
   \int_{\mathbb{R}^2} (w*v_\circ^2) \,v_\circ^2 \,\ud x\;=\;1\,.
   \end{equation}
Since $f_n \rightharpoonup f$ $H^1$-weakly, then both $\nabla f_n \rightharpoonup \nabla f_\circ$ and $f_n \rightharpoonup f_\circ$ $L^2$-weakly, and the weak lower semi-continuity of $L^2$-norms implies
  \begin{equation*}
   \lim_{n\to\infty}\big(\|\nabla f_n\|_{L^2}^2+\lambda\|f_n\|_{L^2}^2\big)\;\geqslant\;\|\nabla f_\circ\|_{L^2}^2+\lambda\|f_\circ\|_{L^2}^2\,.
  \end{equation*}
   On the other hand, since $\widetilde{\mathcal{W}}^{(w)}_{\alpha,\lambda}$ is the \emph{infimum} of the minimisation problem under consideration, \eqref{eq:arecap} excludes that
  \[
   \lim_{n\to\infty}\big(\|\nabla f_n\|_{L^2}^2+\lambda\|f_n\|_{L^2}^2\big)\;>\;\|\nabla f_\circ\|_{L^2}^2+\lambda\|f_\circ\|_{L^2}^2\,.
  \]
  Therefore, necessarily
  \begin{equation}
   \lim_{n\to\infty}\big(\|\nabla f_n\|_{L^2}^2+\lambda\|f_n\|_{L^2}^2\big)\;=\;\|\nabla f_\circ\|_{L^2}^2+\lambda\|f_\circ\|_{L^2}^2\,,
  \end{equation}
  i.e., $f_n\to f_\circ$ in $H^1(\mathbb{R}^2)$. Together with $c_n\to c_\circ$, and in view of \eqref{eq:equiv-norm}, this means also $v_n\to v_\circ$ in $H^1_{\alpha}(\mathbb{R}^2)$. As a consequence,
  \begin{equation}
   \widetilde{\mathcal{W}}^{(w)}_{\alpha,\lambda}\;=\;\|\nabla f_\circ\|_{L^2}^2+\lambda\|f_\circ\|_{L^2}^2+\frac{c_\circ^2}{\beta_\alpha(\lambda)}\;=\; \mathcal{W}^{(w)}_{\alpha,\lambda}(v_\circ)\,,
  \end{equation}
 that is, $v_\circ$ is a minimiser. And since $v_\circ$ is non-zero, then $\widetilde{\mathcal{W}}^{(w)}_{\alpha,\lambda}>0$. The existence part of Theorem \ref{vs:10} is thus established. For \emph{this} ground state $v_\circ$ also the non-negativity, the spherical symmetry, and the radial decreasing are established, as well as the analogous properties for the regular part $f_\circ$ of $v_\circ$.

\section{Structural properties of the ground states}\label{sec:radial}
We come now to characterising the structural properties of ground states claimed in Theorem \ref{vs:10}, namely the real positivity (up to constant multiplicative phases), the $H_{\alpha}^2$-regularity, the spherical symmetry, and the strict radial decrease.

We first obtain the explicit Euler-Lagrange equation associated to $\mathcal{W}^{(w)}_{\alpha,\lambda}$.

\begin{lemma}\label{le:EL}
Let $\alpha\in\mathbb{R}$ and $\lambda>|e_{\alpha}|$, and let $w$ satisfy \eqref{ass_w}. Then the functional $\mathcal{W}^{(w)}_{\alpha,\lambda}$ defined by \eqref{eq.wp1} is of class $\mathcal{C}^1$ on $H_{\alpha}^1(\R^2)\setminus\{0\}$. %with Fr\'echet derivative given by
%\begin{equation}\label{freder}
	%D_v\mathcal{W}^{(w)}_{\alpha,\lambda}=2\,\big(\!	-\Delta_{\alpha}v+\lambda v - \Lambda_{\alpha,\lambda}^{(w)}(v) (w* \left|v \right|^2)v\big)\displaystyle\left(\int_{\mathbb{R}^2}(w\ast|v|^2)|v|^2 \ud x\right)^{\!-\frac{1}{2}}
%\end{equation}
%for $v\neq 0$, where
%$$\Lambda_{\alpha,\lambda}^{(w)}(v)\;:=\;\frac{(-\Delta_{\alpha})[v] +\lambda \|v\|_{L^2}^2}{\;\displaystyle\int_{\mathbb{R}^2}(w\ast|v|^2)|v|^2 \ud x}\,.$$
Moreover, every minimiser $v_{\alpha}$ of \eqref{200} satisfies the Euler-Lagrange equation
\begin{equation}\label{eq.rsr1}
	-\Delta_{\alpha}v_{\alpha}+\lambda v_{\alpha} - \Lambda_{\alpha,\lambda}^{(w)}(v_{\alpha}) (w* \left|v_{\alpha} \right|^2)v_{\alpha}\;=\;0
\end{equation}
as an identity in $H_{\alpha}^{-1}(\R^2)$, where $\Lambda_{\alpha,\lambda}^{(w)}:H_{\alpha}^1(\R^2)\setminus\{0\}\to\R$ is defined by \eqref{eq:Lambdav}.
\end{lemma}
\begin{proof}
Let us set $L:=-\Delta_{\alpha}+\lambda$, which is a positive, self-adjoint operator on $L^2(\R^2)$, with form domain $H_{\alpha}^1(\R^2)$. A direct computation using \eqref{eq:energynormH1alpha-bis} yields
$$L[v+h]-L[v]-2\,\mathfrak{Re}\langle L^{\frac12}v,L^{\frac12}h\rangle_{L^2}\;=\;\|h\|_{H_{\alpha}^1}^2+(e_{\alpha}+\lambda-1)\|h\|_{L^2}^2\;=\;o(\|h\|_{H_{\alpha}^1})$$
for every $v,h\in H_{\alpha}^1(\R^2)$, thereby implying that the functional $v\mapsto L[v]$ belongs to $\mathcal{C}^1(H^1_{\alpha}(\R^2),\R)$, with Fr\'echet derivative $D_v(L[\cdot])\in H_{\alpha}^{-1}(\R^2)$ given by
\begin{equation*}\tag{a}\label{derFL}
\big(D_v(L[\cdot])\big)(h)\;=\;2\,\mathfrak{Re}\langle L^{\frac12}v,L^{\frac12}h\rangle_{L^2}\,,\qquad h\in H_{\alpha}^1(\R^2)\,.
\end{equation*}
Next, consider the functional
$$G(v)\;:=\;\frac14\int_{\R^2} (w*|v|^2)|v|^2\,\ud x\,,\quad v\in H_{\alpha}^1(\R^2)\,.$$
Given $v,h\in H_{\alpha}^1(\R^2)$, a direct computation (using Lemma \ref{le:change.w}) and the bound \eqref{quartic} yield
\[
 \begin{split}
  & G(v+h)-G(v)-\int_{\R^2} (w*|v|^2)\,\mathfrak{Re}(v\overline{h})\,\ud x \\
  &\qquad =\;\int_{\R^2} \big(w*\mathfrak{Re}(v\overline{h})\big)\big(|h|^2+\mathfrak{Re}(v\overline{h})\big)\,\ud x \\
  &\qquad\qquad +\frac12\int_{\R^2}(w*|v|^2)|h|^2\,\ud x+\frac14\int_{\R^2}(w*|h|^2)|h|^2\,\ud x \\
  &\qquad\lesssim\;\|h\|_{H_{\alpha}^1}^2 \big(\|v\|_{H_{\alpha}^1}+\|h\|_{H_{\alpha}^1}\big)^2
 \end{split}
\]
and
\[
 \bigg|\int_{\R^2} (w*|v|^2)\,\mathfrak{Re}(v\overline{h})\,\ud x\bigg|\;\lesssim\,\|v\|_{H_{\alpha}^1}^3\|h\|_{H_{\alpha}^1}\,.
\]
This shows that the functional $G$ is of class $\mathcal{C}^1$ on $H_{\alpha}^1(\R^2)$, with Fr\'echet derivative $D_vG\in H_{\alpha}^{-1}(\R^2)$ given by
\begin{equation*}\tag{b}\label{derFG}
(D_vG)(h)\;=\;\int_{\R^2} (w*|v|^2)\,\mathfrak{Re}(v\overline{h})\,\ud x\,,\qquad h\in H_{\alpha}^1(\R^2)\,.
\end{equation*}
Since $\mathcal{W}^{(w)}_{\alpha,\lambda}(v)=\frac{1}{2}L[v]G(v)^{-\frac{1}{2}}$ and $G(v)$ does not vanish for non-zero $v$'s, then $\mathcal{W}^{(w)}_{\alpha,\lambda}$ is of class $\mathcal{C}^1$ on $H_{\alpha}^1(\R^2)\setminus\{0\}$, and the Leibniz rule, together with \eqref{derFL}, \eqref{derFG}, and the identity $\Lambda_{\alpha,\lambda}^{(w)}(v)=\frac14L[v]G(v)^{-1}$, yields
\begin{equation*}\tag{c}\label{eq:frederi}
\begin{split}
  \big(D_v\mathcal{W}^{(w)}_{\alpha,\lambda}\big)&(h) \;=\; \frac{1}{2}\big(D_v(L[\cdot])\big)(h)G(v)^{-\frac{1}{2}}-\frac{1}{4}L[v]G(v)^{-\frac{3}{2}}(D_vG)(h) \\
  & =\;G(v)^{-\frac{1}{2}}\Big(\mathfrak{Re}\langle L^{\frac12}v,L^{\frac12}h\rangle_{L^2}-\Lambda_{\alpha,\lambda}^{(w)}(v)\int_{\R^2} (w*|v|^2)\,\mathfrak{Re}(v\overline{h})\,\ud x\Big)\,.
\end{split}
\end{equation*}
Suppose now that $v_{\alpha}$ is a minimiser for the problem \eqref{200}. Then $v_{\alpha}$ is a critical point for the functional $\mathcal{W}^{(w)}_{\alpha,\lambda}$, which implies  $D_v\mathcal{W}^{(w)}_{\alpha,\lambda}\big|_{v=v_{\alpha}}\equiv 0$ and hence also the trivial identity
\begin{equation*}\tag{d}\label{polar_critical}
\big(D_v\mathcal{W}^{(w)}_{\alpha,\lambda}\big|_{v=v_{\alpha}}\big)(h)+\ii\big(D_v\mathcal{W}^{(w)}_{\alpha,\lambda}\,|_{v=v_{\alpha}}\big)(\ii h)\;=\;0\qquad\,\forall h\in H_{\alpha}^1(\R^2)\,.
\end{equation*}
Plugging \eqref{eq:frederi} into \eqref{polar_critical} yields
\begin{equation*}\tag{e}\label{form_intermedia}
	\begin{split}
	&\mathfrak{Re}\langle L^{\frac12}v_{\alpha},L^{\frac12}h\rangle_{L^2}+\ii\,\mathfrak{Re}\langle L^{\frac12}v_{\alpha},L^{\frac12}(\ii h)\rangle_{L^2}\\
&\;\;-\Lambda_{\alpha,\lambda}^{(w)}(v_{\alpha})\int_{\R^2}(w*|v_{\alpha}|^2)\big(\mathfrak{Re}(v_{\alpha}\overline{h})+\ii\,\mathfrak{Re}(v_{\alpha}\overline{\ii h})\big)\,\ud x\;=\;0\quad\forall h\in H_{\alpha}^1(\R^2)\,.
	\end{split}
\end{equation*}
On the other hand, on account of the fact that $-\Delta_{\alpha}$ is a real operator, and so too is therefore $L^{\frac12}$,
\[
 \langle L^{\frac12}v_{\alpha},L^{\frac12}h\rangle_{L^2}\;=\;\mathfrak{Re}\langle L^{\frac12}v_{\alpha},L^{\frac12}h\rangle_{L^2}+\ii\,\mathfrak{Re}\langle L^{\frac12}v_{\alpha},L^{\frac12}(\ii h)\rangle_{L^2}\,,
\]
and moreover,
\[
 v_{\alpha}\overline{h}\;=\;\mathfrak{Re}(v_{\alpha}\overline{h})+\ii\,\mathfrak{Re}(v_{\alpha}\overline{\ii h})\,.
\]
By means of the latter two identities, \eqref{form_intermedia} now gives
\begin{equation*}\tag{f}\label{form_EL}
\langle L^{\frac12}v_{\alpha},L^{\frac12}h\rangle_{L^2}-\Lambda_{\alpha,\lambda}^{(w)}(v_{\alpha})\int_{\R^2} (w*|v_{\alpha}|^2)\,v_{\alpha}\,\overline{h}\,\ud x\;=\;0\qquad\,\forall h\in H_{\alpha}^1(\R^2)\,.
\end{equation*}
The integral in \eqref{form_EL} can be re-interpreted as a duality product since, owing to \eqref{quartic},
$$\bigg|\int_{\R^2} (w*|v_{\alpha}|^2)\,v_{\alpha}\,\overline{h}\,\ud x\bigg|\;\lesssim\; \|v_{\alpha}\|_{H_{\alpha}^1}^3\|h\|_{H_{\alpha}^1}\,,$$
and moreover
\[
\begin{split}
  \langle L^{\frac12}v_{\alpha},L^{\frac12}h\rangle_{L^2}\;=\;L[v_\alpha,h]\;&=\;(-\Delta_\alpha)[v_\alpha,h]+\lambda\langle v_\alpha,h\rangle_{L^2} \\
  &=\;\langle (-\Delta_{\alpha}+\lambda)v_{\alpha} , h\rangle_{H_{\alpha}^{-1},H_{\alpha}^1}
\end{split}
\]
(as customary, $L[\cdot,\cdot]$ and $(-\Delta_\alpha)[\cdot,\cdot]$ are the sesquilinear forms associated, respectively, with the quadratic forms $L[\cdot]$ and $(-\Delta_\alpha)[\cdot]$ and obtained from the latter by polarisation), so that \eqref{form_EL} reads also
\[\tag{g}\label{eq:g-added}
\begin{split}
  \langle (-\Delta_{\alpha}+\lambda)v_{\alpha} , h\rangle_{H_{\alpha}^{-1},H_{\alpha}^1}&-\Lambda_{\alpha,\lambda}^{(w)}(v_{\alpha})\langle(w*|v_{\alpha}|^2)v_{\alpha},h\rangle_{H_{\alpha}^{-1},H_{\alpha}^1}\;=\;0 \\
  &\forall h\in H_{\alpha}^1(\R^2)\,.
\end{split}
\]
In other words, the \emph{anti-linear} maps
\[
 \begin{split}
  H^1_\alpha(\mathbb{R}^2)\ni h&\mapsto\langle L^{\frac12}v_{\alpha},L^{\frac12}h\rangle_{L^2}\,, \\
  H^1_\alpha(\mathbb{R}^2)\ni h&\mapsto\int_{\R^2} (w*|v_{\alpha}|^2)\,v_{\alpha}\,\overline{h}\,\ud x
 \end{split}
\]
are canonically identified, as elements in $H^{-1}_\alpha(\mathbb{R}^2)$, respectively with $(-\Delta_{\alpha}+\lambda)v_{\alpha}$ and $(w*|v_{\alpha}|^2)v_{\alpha}$, via $(-\Delta_{\alpha}+\lambda)v_{\alpha}\cong\langle (-\Delta_{\alpha}+\lambda)v_{\alpha},\cdot\rangle_{H_{\alpha}^{-1},H_{\alpha}^1}$ and $(w*|v_{\alpha}|^2)v_{\alpha}\cong\langle (w*|v_{\alpha}|^2)v_{\alpha},\cdot\rangle_{H_{\alpha}^{-1},H_{\alpha}^1}$. Now, \eqref{eq:g-added} is precisely \eqref{eq.rsr1} meant as an identity in $H_{\alpha}^{-1}(\R^2)$.
\end{proof}

Next, we establish the $H_{\alpha}^2$-regularity of the minimisers.

\begin{lemma}\label{le:H2}
Let $\alpha\in\mathbb{R}$ and $\lambda>|e_{\alpha}|$, and let $w$ satisfy \eqref{ass_w}. Suppose that $v_{\alpha}$ is a $H^1_\alpha$-minimiser for the problem \eqref{200}. Then $v_{\alpha}\in H^2_{\alpha}(\mathbb{R}^2)\setminus\{0\}$. In particular,
\[
v_{\alpha}=f_{\alpha}+\frac{f_{\alpha}(0)}{\beta_{\alpha}(\lambda)}\,\mathcal{G}_{\lambda}
\]
for some $f_{\alpha}\in H^2(\R^2)\setminus\{0\}$.
\end{lemma}

\begin{proof}
Let $\varepsilon>0$. The Yosida approximation
$$v_{\alpha}^{(\varepsilon)} \;:=\; (1+\varepsilon(-\Delta_{\alpha} +\lambda))^{-1} v_{\alpha}\,\in\, H_{\alpha}^2(\mathbb{R}^2)$$
satisfies, by construction,
\[
 v_{\alpha}^{(\varepsilon)}\,\xrightarrow{\;\varepsilon\downarrow 0\;}\,v_\alpha\qquad\textrm{in }L^2(\mathbb{R}^2)
\]
and satisfies also, on account of \eqref{eq.rsr1},
\begin{equation*}\tag{*}\label{eq:L2H2}
	(-\Delta_{\alpha} +\lambda)\,v_{\alpha}^{(\varepsilon)}  \;=\;\Lambda_{\alpha,\lambda}^{(w)}(v_{\alpha})\,  (1+\varepsilon(-\Delta_{\alpha} +\lambda))^{-1}( (w * |v_{\alpha}|^2) v_{\alpha} )\,.
\end{equation*}
Here \eqref{eq:L2H2} would be meant a priori in the $H^{-1}_\alpha$-sense, because \eqref{eq.rsr1} holds in $H^{-1}_\alpha(\mathbb{R}^2$); however,  $(w * |v_{\alpha}|^2) v_{\alpha} \in L^2(\mathbb{R}^2)$, owing to \eqref{eq:L2Hartree}, therefore both sides of \eqref{eq:L2H2} actually belong to $L^2(\mathbb{R}^2)$. Moreover, as $\varepsilon\downarrow 0$, the r.h.s.~of \eqref{eq:L2H2} $L^2$-converges to $\Lambda_{\alpha,\lambda}^{(w)}(w * |v_{\alpha}|^2) v_{\alpha}$, thus implying
\[
 (-\Delta_{\alpha} +\lambda)\,v_{\alpha}^{(\varepsilon)}\,\xrightarrow{\;\varepsilon\downarrow 0\;}\,\Lambda_{\alpha,\lambda}^{(w)}(v_{\alpha})(w * |v_{\alpha}|^2) v_{\alpha}\qquad\textrm{in }L^2(\mathbb{R}^2)\,.
\]
We then deduce that $v_\alpha^{(\varepsilon)}$ converges in $H^2_\alpha(\mathbb{R}^2)$, and its $H^2_\alpha$-limit must necessarily coincide with its $L^2$-limit $v_\alpha$ above. The conclusion is that $v_\alpha\in H^2_\alpha(\mathbb{R}^2)$. In turn, this fact implies the canonical decomposition of $v_\alpha$ stated in the Lemma.
\end{proof}

 We finally come to demonstrating the remaining properties of ground states, which then completes the proof of Theorem \ref{vs:10}.  Observe that, on account of Lemma \ref{le:H2}, we can restrict the analysis to $H_{\alpha}^2$-minimisers.

\begin{proposition}\label{pr:machi}
	Let $\alpha\in\mathbb{R}$ and $\lambda>|e_{\alpha}|$, assume that $w$ satisfy \eqref{ass_w}, and let
	$$v_{\alpha}\;=\;f_{\alpha}+\frac{f_{\alpha}(0)}{\beta_{\alpha}(\lambda)}\,\mathcal{G}_{\lambda}$$
	be a minimiser for the problem \eqref{200}, for some $f_{\alpha}\in H^2(\R^2)\setminus\{0\}$. Then:
	\begin{enumerate}
	 \item[(i)] $|f_{\alpha}|$ is spherically symmetric, strictly positive, and strictly radially decreasing in $|x|$;
\item[(ii)] $v_{\alpha}$ has constant phase, namely,
	\begin{equation}\label{vcop}
		v_{\alpha}\;=\; e^{\ii\theta_{\alpha}}\Big(\,|f_{\alpha}|  +\frac{|f_{\alpha} (0)|}{\beta_{\alpha} (\lambda)} \mathcal{G}_\lambda\Big)
		\end{equation}
for some $\theta_{\alpha}\in[0,2\pi)$.
	\end{enumerate}
\end{proposition}

\begin{proof}
(i) On account of Corollary \ref{cor:loweringWeinstein}, and since $v_\alpha$ is a minimiser,
\begin{equation*}\tag{a}\label{eq:allequal}
 \begin{split}
  &\mathcal{W}^{(w)}_{\alpha,\lambda}\Big(|f_{\alpha}|^*+\frac{|f_{\alpha}|(0)}{\beta_{\alpha}(\lambda)}\,\mathcal{G}_\lambda\Big) \;=\;\mathcal{W}^{(w)}_{\alpha,\lambda}\Big(|f_{\alpha}|+\frac{|f_{\alpha}|(0)}{\beta_{\alpha}(\lambda)}\,\mathcal{G}_\lambda\Big) \\
  &\qquad\qquad\;=\;\mathcal{W}^{(w)}_{\alpha,\lambda}\Big(f_{\alpha}+\frac{f_{\alpha}(0)}{\beta_{\alpha}(\lambda)}\,\mathcal{G}_\lambda\Big)\,,
 \end{split}
  \end{equation*}
that is, $|f_{\alpha}|^*+\frac{|f_{\alpha}|(0)}{\beta_{\alpha}(\lambda)}\,\mathcal{G}_{\lambda}$ and $|f_{\alpha}|+\frac{|f_{\alpha}|(0)}{\beta_{\alpha}(\lambda)}\,\mathcal{G}_{\lambda}$ are minimisers as well (recall that $\beta_{\alpha}(\lambda)>0$ for $\lambda>|e_{\alpha}|$). As such, they both belong to $H^2_\alpha(\mathbb{R}^2)\setminus\{0\}$ (Lemma \ref{le:H2}), meaning that $|f_{\alpha}|,|f_{\alpha}|^*\in H^2(\R^2)\setminus\{0\}$ and $|f_{\alpha}|^*(0)=|f_{\alpha}|(0)$. In particular, the function $|f_{\alpha}|^*+\frac{|f_{\alpha}|^*(0)}{\beta_{\alpha}(\lambda)}\,\mathcal{G}_{\lambda}$ is a minimiser too. Moreover, $|f_{\alpha}|^*(0)>0$, otherwise the radial decrease in $|x|$ of $|f_{\alpha}|^*$ would imply $|f_{\alpha}|^*\equiv 0$, a contradiction.

 Now, the first identity in \eqref{eq:allequal} reads
\[
 \frac{\,\|\nabla|f_\alpha|^*\|_{L^2}^2+\lambda\,\|\,|f_\alpha|^*\,\|_{L^2}^2\,}{\mathfrak{g}(|f_\alpha|^*)}\;=\;\frac{\,\|\nabla|f_\alpha|\|_{L^2}^2+\lambda\,\|\,|f_\alpha|\,\|_{L^2}^2\,}{\mathfrak{g}(|f_\alpha|)}\,,
\]
where we used $|f_{\alpha}|^*(0)=|f_{\alpha}|(0)$, \eqref{eq:opaction} for the numerators, and the short-hand
\[
 \mathfrak{g}(\varphi)\;:=\;
 \bigg(\int_{\mathbb{R}^2}  \Big( w*\Big|\varphi+\frac{\varphi(0)}{\beta_{\alpha}(\lambda)}\,\mathcal{G}_\lambda\Big|^2\Big)\Big|\varphi+\frac{\varphi(0)}{\beta_{\alpha}(\lambda)}\,\mathcal{G}_\lambda\Big|^2\,\ud x\bigg)^{\!\frac{1}{2}}
\]
for the denominators. In the latter identity, clearly, $\|\,|f_\alpha|^*\,\|_{L^2}=\|\,|f_\alpha|\,\|_{L^2}$, and moreover $\mathfrak{g}(|f_\alpha|)\leqslant\mathfrak{g}(|f_\alpha|^*)$ (as a consequence of Lemma \ref{lbl12} and the property $|f_{\alpha}|^*(0)=|f_{\alpha}|(0)$); therefore, necessarily $\|\,\nabla|f_\alpha|^*\|_{L^2}\geqslant \|\,\nabla|f_\alpha|\,\|_{L^2}$, which implies, for compatibility with the  P\'{o}lya-Szeg\H{o} inequality \eqref{eq:PSz},
\[\tag{b}\label{eq:PSidentity}
 \|\nabla|f_\alpha|^*\|_{L^2}\;=\; \|\nabla|f_\alpha|\,\|_{L^2}\,.
\]

On the other hand, the Euler-Lagrange equation \eqref{eq.rsr1} for the minimiser $|f_{\alpha}|^*+\frac{|f_{\alpha}|^*(0)}{\beta_{\alpha}(\lambda)}\,\mathcal{G}_{\lambda}$ reads
\[\tag{c}\label{eq:ELspecial}
 -\Delta|f_\alpha|^*+\lambda|f_\alpha|^*\;=\;\Big( w*\Big(|f_\alpha|^*+\frac{|f_\alpha|^*(0)}{\beta_{\alpha}(\lambda)}\,\mathcal{G}_\lambda\Big)^2\Big)\Big(|f_\alpha|^*+\frac{|f_\alpha|^*(0)}{\beta_{\alpha}(\lambda)}\,\mathcal{G}_\lambda\Big)\,\widetilde{\Lambda}\,,
\]
having set $\widetilde{\Lambda}:=\Lambda_{\alpha,\lambda}^{(w)}\big(|f_{\alpha}|^*+\frac{|f_{\alpha}|^*(0)}{\beta_{\alpha}(\lambda)}\,\mathcal{G}_{\lambda}\big)>0$ and used again \eqref{eq:opaction} in the l.h.s.

Adapting the reasoning from \cite[Lemma 4.5]{Fukaya-Georgiev-Ikeda-2021}, we see that if, for contradiction, $|f_\alpha|^*$ was not \emph{strictly} radially decreasing, and instead was constant on some annulus $0\leqslant r_1<|x|<r_2$, then on such a region \eqref{eq:ELspecial} would imply
\[
 \textrm{constant}\;=\;|f_\alpha|^*\;=\;\textrm{r.h.s.~of \eqref{eq:ELspecial}}\;=\;\textrm{radially strictly decreasing function}\,.
\]
Indeed, in the r.h.s.~of \eqref{eq:ELspecial} $|f_\alpha|^*+\frac{|f_\alpha|^*(0)}{\beta_{\alpha}(\lambda)}\,\mathcal{G}_\lambda$ is \emph{strictly} decreasing with $|x|$ (to claim this, it is crucial that $|f_\alpha|^*(0)>0$), and $w*\big(|f_\alpha|^*+\frac{|f_\alpha|^*(0)}{\beta_{\alpha}(\lambda)}\,\mathcal{G}_\lambda\big)^2$ is non-increasing with $|x|$. To avoid the above contradiction, necessarily $|f_\alpha|^*$ is strictly radially decreasing. As a consequence, setting
\[
 \begin{split}
   \mathcal{B}\,&:=\,\{x\in\mathbb{R}^2\,|\,0<|f_\alpha|^*(x)<\||f_\alpha|^*\|_{L^\infty}=|f_\alpha|^*(0)\}\,,  \\
  \mathcal{C}\,&:=\,\{x\in\mathbb{R}^2\,|\,\nabla|f_\alpha|^*(x)=0\}\,,
 \end{split}
\]
 one has $|\mathcal{C}|=0$ (in the sense of the Lebesgue measure $|\cdot|$) and therefore also
 \[\tag{d}\label{eq:zeromeasure}
 \big|\mathcal{B}\cap \mathcal{C}\big|\;=\;0\,.
\]
%
% the set $\mathcal{C}:=\{x\in\mathbb{R}^2\,|\,\nabla|f_\alpha|^*(x)=0\}$ has zero Lebesgue measure, and so does its symmetric re-arrangement $\mathcal{C}^*$. Thus,
% % \[\tag{d}\label{eq:zeromeasure}
% %  \big|\,\mathcal{C}^*\cap\{x\in\mathbb{R}^2\,|\,0<|f_\alpha|^*(x)<\||f_\alpha|^*\|_{L^\infty}=|f_\alpha|^*(0)\}\,\big|\;=\;0
% % \]
% (with the customary notation $|\mathcal{A}|$ for the Lebesgue measure of the set $\mathcal{A}$).
%
An obvious further consequence of the strict radial decrease of $|f_\alpha|^*$ is that $|f_\alpha|^*$ is also strictly positive.

 Condition \eqref{eq:zeromeasure} above, the non-negativity of $|f_\alpha|^*$, and its vanishing at infinity in the sense of finite measure of the $t$-level sets $\{x\in\mathbb{R}^2\,|\,|f_\alpha|^*(x)>t\}$ for all $t>0$ (which follows from $|f_\alpha|^*\in H^2(\mathbb{R}^2)\supset L^2(\mathbb{R}^2)$), are precisely the hypotheses of \cite[Theorem 1.1]{Burchard-Ferone-2015}, which then implies
%Conditions \eqref{eq:PSidentity} and \eqref{eq:zeromeasure} above allow to appeal to the classical result \cite[Theorem~1.1]{Brothers-Ziemer-1988} (picking $u=|f_\alpha|$, $A(\xi)=\xi^2$, and $p=4$, in the notation therein), which implies
 the existence of $x_0\in\mathbb{R}^2$ such that $|f_\alpha|(x)=|f_\alpha|^*(x-x_0)$ for \emph{every} $x\in\mathbb{R}^2$ (recall that both $|f_\alpha|$ and $|f_\alpha|^*$ are \emph{continuous}, being two $H^2$-functions). The property $|f_{\alpha}|(0)=|f_{\alpha}|^*(0)$ and the strict radial decrease of $|f_\alpha|^*$ then yield
\[
 |f_\alpha|(0)\,=\,|f_\alpha|^*(-x_0)\,<\,|f_\alpha|^*(0)\,=\,|f_\alpha|(0)\qquad \textrm{when }\;x_0\neq 0\,,
\]
a contradiction. Therefore, $x_0=0$. We thus conclude that $|f_\alpha|=|f_\alpha|^*$. Since $|f_\alpha|^*$ is spherically symmetric, strictly positive, and strictly radially decreasing, so too is then $|f_\alpha|$. Part (i) is thus proved.

(ii) The second identity in \eqref{eq:allequal} (on account of \eqref{eq:opaction}) reads
\[
 \frac{\,\|\nabla|f_\alpha|\|_{L^2}^2+\lambda\|f_\alpha\|_{L^2}^2\,}{\mathfrak{g}(|f_\alpha|)}\;=\;\frac{\,\|\nabla f_\alpha\|_{L^2}^2+\lambda\|f_\alpha\|_{L^2}^2\,}{\mathfrak{g}(f_\alpha)}\,.
\]
In view of the estimates $ \|\nabla|f_\alpha|\|_{L^2}\leqslant\|\nabla f_\alpha\|_{L^2}$ and $\mathfrak{g}(|f_\alpha|)\geqslant\mathfrak{g}(f_\alpha)$, we then necessarily have
\[\tag{e}\label{eq:equalgradients}
 \|\nabla|f_\alpha|\|_{L^2}\;=\;\|\nabla f_\alpha\|_{L^2}\,.
\]
Re-writing  $f_{\alpha}(x)=|f_{\alpha}(x)|e^{\ii\theta_{\alpha}(x)}$ (the phase $\theta_{\alpha}$ is well-defined since $|f_{\alpha}|>0$), and using $|\nabla f_\alpha|^2=|\nabla|f_\alpha||^2+|f_{\alpha}|^2|\nabla\theta_\alpha|^2$, condition \eqref{eq:equalgradients} and the strict positivity of $|f_{\alpha}|$ imply $\nabla\theta_{\alpha}=0$, i.e., $\theta_\alpha$ is constant (and real). We then have
$$v_{\alpha}\;=\;f_{\alpha}  +\frac{f_{\alpha} (0)}{\beta_{\alpha} (\lambda)} \mathcal{G}_\lambda\;=\;e^{i\theta_{\alpha}}\Big(|f_{\alpha}|  +\frac{|f_{\alpha} (0)|}{\beta_{\alpha} (\lambda)} \mathcal{G}_\lambda\Big)\,,$$
which establishes \eqref{vcop} and concludes the proof.
\end{proof}

 \begin{remark}
  In the above proof of Proposition \ref{pr:machi}(i), we could have deduced the translation property $|f_\alpha|(x)=|f_\alpha|^*(x-x_0)$ also appealing to the classical result \cite[Theorem~1.1]{Brothers-Ziemer-1988}, of which \cite[Theorem 1.1]{Burchard-Ferone-2015} that we used here is a later generalisation. Indeed, since $|\mathcal{C}|=0$, then also $|\mathcal{C}^*|=0$ ($\mathcal{C}^*$ is the symmetric re-arrangement of the set $\mathcal{C}$), whence $|\mathcal{B}\cap\mathcal{C}^*|=0$. The latter fact and condition \eqref{eq:PSidentity} are the hypotheses of \cite[Theorem~1.1]{Brothers-Ziemer-1988} (picking $u=|f_\alpha|$, $A(\xi)=\xi^2$, and $p=4$, in the notation therein), which implies the translation property above, modulo the fact that $u=|f_\alpha|$ does \emph{not} have compact support here. Thus, the fair application of \cite[Theorem~1.1]{Brothers-Ziemer-1988} would have required an additional check that one can indeed go back to the original assumptions (compactness of the support) prescribed in \cite[Theorem~1.1]{Brothers-Ziemer-1988} by smooth localisation of $|f_\alpha|$ over closed balls, an analysis on which we omit the detail. Clearly, \cite[Theorem~1.1]{Brothers-Ziemer-1988} by-passes all that.
 \end{remark}

\section{Local well-posedness in the energy space}\label{sec:LWP}

In this Section we establish the local well-posedness in $H^1_\alpha(\mathbb{R}^2)$ of the initial value problem \eqref{cp:shmain}, as stated in Theorem \ref{th:main_ex}. Here focusing and defocusing case are treated on an equal footing: their different nature only emerges in the long-time behaviour.

 In preparation for that, let us concisely review the abstract solution theory in the energy space for a class of semi-linear Schr\"{o}dinger equations. The result imported in Proposition \ref{pr:ass} below is based on a suitable generalisation, proposed in \cite{Okazawa-Suzuki-Yokota-2012}, of the Cazenave existence method \cite[Section 3.3]{cazenave}.

Let $S$ be a lower semi-bounded self-adjoint operator on a Hilbert space $X$, with greatest lower bound $\mathfrak{m}(S)$. The associated form domain is the energy space
$$X_S\;:=\;\mathcal{D}((1-\mathfrak{m}(S)+S)^{\frac{1}{2}})\,,$$
a Hilbert space whose scalar product is induced by the norm
\[
 \|\psi\|_{X_S}\;:=\;\big\|(1-\mathfrak{m}(S)+S)^{\frac{1}{2}}\psi\big\|_{X}\,.
\]
$X_S^*$ shall denote, as customary, the topological dual of $X_S$, thus itself equipped with a natural Hilbert space structure, in particular with Hilbert norm $\|\cdot\|_{X_S^*}$. For relevant functions $g:X_S\to X_S^*$ let us consider the solution theory in $X_S$ for the semi-linear equation
\begin{equation}\label{gen_NLS}
	\ii\partial_t \psi\;=\;S \psi+g(\psi)\,.
	\end{equation}
(We write $g(\psi)$ instead of $g\psi$ to emphasise that $g$ is not necessarily linear.)
To this aim, one poses the following assumptions on $g$.
\begin{enumerate}
 \item[\textbf{(G1)}] $g$ is Lipschitz continuous on bounded sets of $X_S$.
 %, that is, for all $M>0$
%  for all $M>0$ there exists $C(M)>0$ such that
% \[\|g(\psi_1)-g(\psi_2)\|_{X_S^*}
% \le C(M)\|\psi_1-\psi_2\|_{X_S}\quad
% \forall \psi_1,\psi_2\in X_S
% \text{ with $\|\psi_1\|_{X_S}$,\,$\|\psi_2\|_{X_S}\le M$}. \]
 \item[\textbf{(G2)}] $g=G'$ for some $G\in C^1(X_S,\R)$, understanding $G'$ as computed with respect to the \emph{real} Hilbert space duality structure of $X_S$ and $X_S^*$, that is,
 \[
  \begin{split}
   & \forall \psi\in X_S\,\;\forall\varepsilon>0\,\;\;\exists\,\delta_{\psi,\varepsilon}>0 \;\;\textrm{such that}     \\
   &\big|G(\psi+\xi)-G(\psi)-\mathfrak{Re}\big(\langle g(\psi),\xi\rangle_{X_S^*,X_S}\big)\big|\,\leqslant\,\varepsilon\|\xi\|_{X_S} \\
   &\forall \xi\in X_S\;\textrm{ with }\; \|\xi\|_{X_S}\,\leqslant\,\delta_{\psi,\varepsilon}\,,
  \end{split}
 \]
  where $\langle\cdot,\cdot\rangle_{X_S^*,X_S}$ denotes the (complex-valued) duality pairing between $X_S^*$ and $X_S$, with the convention of being linear in the first entry and anti-linear in the second.
 \item[\textbf{(G3)}] $G$ satisfies
 \[
  \begin{split}
  & \forall\,\delta,M>0\;\;\exists\,C_{\delta,M}>0\;\;\textrm{such that} \\
  & |G(\psi_1)-G(\psi_2)|\;\leqslant\;\delta+C_{\delta,M}\|\psi_1-\psi_2\|_{X} \\
  & \forall\, \psi_1,\psi_2\in X_S\,\textrm{ with }\,\|\psi_1\|_{X_S},\|\psi_2\|_{X_S}\,\leqslant\, M\,.
  \end{split}
 \]
 \item[\textbf{(G4)}] $g$ satisfies $\mathfrak{Im}\big(\langle g(\psi),\psi\rangle_{X_S^*,X_S}\big)=0$.
 \item[\textbf{(G5)}] $g$ is a weakly closed map, i.e., for any $\psi$ and any sequence $(\psi_n)_{n\in\N}$ in $X_S$ satisfying
 \[
  \begin{cases}
   \;\;\;\;\;\,\psi_n\,\xrightarrow{n\to\infty}\,\psi & \textrm{weakly in $X_S$,} \\
   \;g(\psi_n)\,\xrightarrow{n\to\infty}\,\phi & \textrm{weakly in $X_S^*$}
  \end{cases}
 \]
  for some $\phi\in X_S^*$,  one necessarily has $\phi=g(\psi)$. (The weak convergence is meant here in the sense of the Banach space weak topology.)

%%%%%%%%%%%%%%%% OLD G5
%  \item[\textbf{(G5)}] For any $\psi$ and any bounded sequence $(\psi_n)_{n\in\N}$ in $L^\infty(I,X_S)$, with $I$ a bounded open interval, satisfying
%  \[
%   \begin{cases}
%    \;\psi_n(t,\cdot)\,\xrightarrow{n\to\infty}\,\psi(t,\cdot) & \textrm{weakly in $X_S$ for a.e.~$t\in I$}, \\
%    \;\;\;g(\psi_n)\,\xrightarrow{n\to\infty}\,\phi & \textrm{weakly-$*$ in $L^\infty(I,X_S^*)$}
%   \end{cases}
%  \]
%  for some $\phi\in L^\infty(I,X_S^*)$, one has
%  \[
%  \lim_{n\to\infty}\int_I\mathfrak{Im}\big(\langle g(\psi_n(t,\cdot)),\psi_n(t,\cdot)\rangle_{X_S^*,X_S}\big)\,\ud t\;=\;\int_I \mathfrak{Im}\big(\langle \phi(t),\psi(t)\rangle_{X_S^*,X_S}\big)\,\ud t\,.
% \]
 \item[\textbf{(G6)}] If, for some $T>0$, the initial value problem
 \begin{equation}\label{cp:shg}
	\begin{cases}
		\;\;\;\;\ii\partial_t\psi\,=\,S\psi +g(\psi)\,,\\
		\;\psi(0,\cdot)\,=\,\psi_0
	\end{cases}
\end{equation}
 is solvable in $ L^\infty((0,T),X_S)\cap W^{1,\infty}((0,T),X_S^*)$, then the solution is unique.
\end{enumerate}

Observe that assumption \textbf{(G6)} expresses the uniqueness of local weak solutions to \eqref{cp:shg}, whereas \textbf{(G4)} is a customary gauge condition.

\begin{proposition}\label{pr:ass}
Let $S$ be a lower semi-bounded self-adjoint operator on the Hilbert space $X$, and let the function $g:X_S\to X_S^*$ satisfy conditions \textbf{(G1)}-\textbf{(G6)} above. Then for any $\psi_0\in X_S$ there exist a maximal time $T_{\mathrm{max}}\equiv T_{\mathrm{max}}(\psi_0)\in(0,+\infty]$ and a unique (maximal) solution $\psi\in\mathcal{C}([0,T_{\mathrm{max}}),X_S)\cap \mathcal{C}^1([0,T_{\mathrm{max}}),X_S^*)$ to the initial value problem \eqref{cp:shg}. Moreover:
\begin{itemize}
	\item[(i)] if $T_{\mathrm{max}}<\infty$, then $\displaystyle\lim_{t\uparrow T_{\mathrm{max}}}\|\psi(t)\|_{X_S}=+\infty $;
	\item[(ii)] the quantities $\mathcal{M}(t):=\|\psi(t,\cdot)\|_{X}^2$ and $\mathcal{E}(t):=\frac{1}{2}S[\psi(t,\cdot)]+G(\psi(t,\cdot))$ are constant for every $t\in(0,T_{\mathrm{max}})$;
	\item[(iii)] if, for a sequence $(\psi_0^{(n)})_{n\in\mathbb{N}}$ in $X_S$ one has $\|\psi_0^{(n)}-\psi_0\|_{X_S}\to 0$, and if $T\in(0,T_{\mathrm{max}}(\psi_0))$, then, eventually in $n$, the maximal solution $\psi^{(n)}$ to \eqref{gen_NLS} with initial datum $\psi_0^{(n)}$ is defined on $[0,T]$ and satisfies $\psi^{(n)}\to\psi$ in $\mathcal{C}([0,T],X_S)$.
\end{itemize}
\end{proposition}

 \begin{proof}
  The existence of a strong solution and property (ii) follow directly from \cite[Lemma 5.3 and Theorem 2.3]{Okazawa-Suzuki-Yokota-2012}, modulo innocent shifts $S\mapsto S-\mathfrak{m}(S)$ and $g(\psi)\mapsto g(\psi)-\frac{\mathfrak{m}(S)}{2}\|\psi\|_X^2$ with respect to the notation therein, where $S$ was only assumed to be positive.
  Then, properties (i) and (iii) can be deduced along the same lines of the standard scheme of \cite[Theorem 3.3.9]{cazenave} (more precisely, \textsc{Step 2} and \textsc{Step 3} therein): for the ordinary (magnetic) Hartree equation studied by (beyond-Strichartz) energy methods this is done, e.g., in \cite[Section 4]{M-2015-nonStrichartzHartree}.
 \end{proof}

 Specialising the above abstract theory to our current setting, we can establish Proposition \ref{pr:local_wp} below, that represents the main result of this Section and will be proved in the following. In doing so, we mirror the analogous approach recently followed in \cite{Fukaya-Georgiev-Ikeda-2021}, where the investigated two-dimensional, singular-perturbed NLS had a pure-power non-linearity, instead of our convolutive non-linearity.

\begin{proposition}\label{pr:local_wp}
	Let $\alpha\in\mathbb{R}$, $\theta=\pm 1$, and let $w$ satisfy \eqref{ass_w}. Then for any $\psi_0\in H^1_{\alpha}(\mathbb{R}^2)$ there exist $T_{\mathrm{max}}\equiv T_{\mathrm{max}}(\psi_0)\in(0,+\infty]$ and a unique (maximal) solution $\psi\in\mathcal{C}([0,T_{\mathrm{max}}),H^1_{\alpha}(\mathbb{R}^2))$ to the initial value problem
	\begin{equation}\label{cp:sh}
	\begin{cases}
		\:\ii\partial_t\psi\,=\,-\Delta_{\alpha}\psi +\theta(w\ast|\psi|^2)\psi\,,\\
		\:\psi(0,\cdot)\,=\,\psi_0\,.
	\end{cases}
\end{equation}
 Moreover:
\begin{itemize}
	\item[(i)] if $T_{\mathrm{max}}<+\infty$, then $\displaystyle\lim_{t\uparrow T_{\mathrm{max}}}\|\psi(t)\|_{H_{\alpha}^1}=+\infty$;	
	\item[(ii)] mass $\mathcal{M}(t)$ and energy $\mathcal{E}(t)$ of $\psi$, as defined in \eqref{def:mass}-\eqref{def:energy}, are constant $\forall\,t\in(0,T_{\mathrm{max}})$;
	\item[(iii)] if, for a sequence $(\psi_0^{(n)})_{n\in\mathbb{N}}$ in $H_{\alpha}^1(\mathbb{R}^3)$, one has $\psi_0^{(n)}\to \psi_0$ in $H_{\alpha}^1(\mathbb{R}^3)$, and if $T\in(0,T_{\mathrm{max}}(\psi_0))$, then, eventually in $n$, the maximal solution $\psi^{(n)}$ to \eqref{singular_hartree-theta} with initial datum $\psi_0^{(n)}$ is defined on $[0,T]$ and satisfies $\psi^{(n)}\to\psi$ in $\mathcal{C}([0,T],H_{\alpha}^1(\mathbb{R}^3))$.
\end{itemize}
\end{proposition}

\begin{remark}
In view of the Strichartz estimates \eqref{eq.STr3}-\eqref{eq.STr4}, it is conceivable to establish Proposition \ref{pr:local_wp} by means of a classical contraction argument (i.e., Kato's method \cite[Section 4.4]{cazenave}). However, this would require a characterisation of the singular-perturbed Sobolev spaces $W_{\alpha}^{1,r}(\mathbb{R}^2)$ adapted to $-\Delta_{\alpha}$, which is currently \emph{unknown} when $r\neq 2$.
\end{remark}

 For the proof of Proposition \ref{pr:local_wp} let us first verify, by means of the inhomogeneous Strichartz estimates \eqref{eq.STr4}, the uniqueness of local weak solutions to \eqref{cp:sh}.

\begin{lemma}\label{le:uni}
	Let $\alpha\in\mathbb{R}$, $\theta=\pm 1$, $T>0$, $\psi_0\in H_{\alpha}^1(\mathbb{R}^2)$, and let $w$ satisfy \eqref{ass_w}. If $\psi_1,\psi_2\in L^\infty((0,T),H_{\alpha}^1(\mathbb{R}^2))\cap W^{1,\infty}((0,T),H_{\alpha}^{-1}(\mathbb{R}^2))$ solve \eqref{cp:sh} for the considered $\psi_0$, then $\psi_1=\psi_2$.
\end{lemma}

\begin{proof}
	The Duhamel formula yields, for every $t\in(0,T)$ and almost all $x\in\mathbb{R}^2$,
	\[
	 \begin{split}
	  & (\psi_1-\psi_2)(t,x)  \\
	  &\quad =\;- \ii \int_0^t \Big(e^{\ii(t-\tau)\Delta_{\alpha} }\big((w*|\psi_1(\tau,\cdot)|^2)\psi_1(\tau,\cdot) - (w*|\psi_2(\tau,\cdot)|^2)\psi_2(\tau,\cdot)\big)\Big)(x) \,\ud \tau\,.
	 \end{split}
	\]
  For $j\in\{1,2\}$ set
$$(s_j,r_j)\,:=\,\begin{cases}
(\infty,2)\,,&\textrm{if }\; p_j>1\,,\\
\;(4,4)\,,& \textrm{if }\; p_j=1\,,
\end{cases}
% \qquad b_j\,:=\,
% \begin{cases}
% 	\frac{2p_j}{p_j-1}\,,&\textrm{if }\;p_i>1\,,\\
% 	\;\;4\,,&\textrm{if }\;p_j=1\,,
% 	\end{cases}
$$
and set also $\mathcal{X}_T:=L_{T}^{s_1}L^{r_1}+L_{T}^{s_2}L^{r_2}$, $ \widetilde{\mathcal{X}}_T:=L_{T}^{s'_1}L^{r'_1}+L_{T}^{s'_2}L^{r'_2}$ (with reference to the notation declared at the end of Section \ref{sec:intro-main}). Recall that $p_1,p_2\in[1,\infty)$ parametrise the assumption $w\in L^{p_1}(\mathbb{R}^2)+L^{p_2}(\mathbb{R}^2)$.
%, and that $p=\min\{p_1,p_2\}$.
For concreteness of presentation, let $(s,r)$ denote for a moment either $(s_1,r_1)$, $(s_2,r_2)$, and similarly, let $p$ stand temporarily for $p_1$ or $p_2$. For generic $T_0\in(0,T)$ we deduce from the above identity for $\psi_1-\psi_2$ that
\[
 \begin{split}
  &\|\psi_1-\psi_2\|_{L^s_{T_0}L^r}\;\leqslant\;C_{T_0}\big\|(w*|\psi_1|^2)\psi_1-(w*|\psi_2|^2)\psi_2\big\|_{L^{s'}_{T_0}L^{r'}} \\
  &\; \lesssim\;C_{T_0}\|w\|_{L^p}  \Big\|\big( \|\psi_1\|^2_{L^b}+\|\psi_2\|^2_{L^b}\big)\|\psi_1-\psi_2\|_{L^r}\Big\|_{L^{s'}_{T_0}}\qquad\qquad\;\; \big(b=\textstyle{(1-\frac{1}{r}-\frac{1}{2p})^{-1}}\big) \\
  &\; \leqslant\;C_{T_0}(T_0)^\sigma\|w\|_{L^p}\big( \|\psi_1\|^2_{L^\infty_{T_0}L^b}+\|\psi_2\|^2_{L^\infty_{T_0}L^b}\big)\|\psi_1-\psi_2\|_{L^s_{T_0}L^r}\;\;\, (\textstyle{\sigma=\frac{2s-1}{2s}})
 \end{split}
\]
having used the Strichartz estimate \eqref{eq.STr4} in the first step, estimate \eqref{eq:dig} in space in the second, and H\"{o}lder inequality in time in the third. Observe that if $(s,r)=(\infty,2)$, then $\sigma=1$ and $b=\frac{2p}{p-1}\in(2,\infty)$ (this is the case where $p>1$), and if instead $(s,r)=(4,4)$, then $\sigma=\frac{7}{8}$ and $b=4$. In either case, $b$ is an admissible spatial index for the embedding \eqref{eq.sob1}, which gives $\|\psi_j\|_{L^\infty_{T_0}L^b}\lesssim \|\psi_j\|_{L^\infty_{T_0}H^1_\alpha}$, $j\in\{1,2\}$. We can therefore conclude
\[
 \|\psi_1-\psi_2\|_{\mathcal{X}_{T_0}}\;\lesssim \;C_{T_0}(T_0)^\sigma\|w\|_{L^{p_1}+L^{p_2}}\big( \|\psi_1\|^2_{L^\infty_{T_0}H^1_\alpha}+\|\psi_2\|^2_{L^\infty_{T_0}H^1_\alpha}\big)\|\psi_1-\psi_2\|_{\mathcal{X}_{T_0}}
\]
for some $\sigma>0$. Recall from \eqref{eq.STr4} that $C_{T}=O(1)$ as $T\downarrow 0$. So, by choosing $T_0$ sufficiently small, we deduce $\|\psi_1-\psi_2\|_{\mathcal{X}_{T_0}}=0$, whence $\psi_1=\psi_2$ point-wise a.e.~on $[0,T_0]\times\mathbb{R}^2$. Covering $[0,T)$ with subsequent intervals $[nT_0,(n+1)T_0]\cap[0,T)$, for increasing $n\in\mathbb{N}_0$, and iterating the above reasoning a sufficient finite number of times, we finally conclude $\psi_1=\psi_2$ for a.e.~$(t,x)\in[0,T)\times\mathbb{R}^2$.
  \end{proof}

%We can prove now the main result of this section.

\begin{proof}[Proof of Proposition \ref{pr:local_wp}]
 This is an application of Proposition \ref{pr:ass} with
	\begin{align*}
		&S\,=\,-\Delta_{\alpha}\,,\quad g(\psi)\,=\,\theta(w*|\psi|^2)\psi\,,
		\\&X\,=\,L^2(\mathbb{R}^2)\,,\quad
		X_S\,=\,H_{\alpha}^1(\mathbb{R}^2),,\quad
		X_S^*\,=\,H_{\alpha}^{-1}(\mathbb{R}^2)\,.
	\end{align*}
 As condition \textbf{(G6)} is ensured in this case by Lemma \ref{le:uni}, the rest of the proof consists of checking the validity of \textbf{(G1)}--\textbf{(G5)}.

 For \textbf{(G1)}, set $q_j:=\frac{3p_j-2}{2p_j}\in(2,6]$, $j\in\{1,2\}$, given the integrability indices $p_1$ and $p_2$ for $w$. Then,
 \[
  \begin{split}
   \|g(\psi_1)&-g(\psi_2)\|_{H_{\alpha}^{-1}}\;\leqslant\;\|g(\psi_1)-g(\psi_2)\|_{L^2} \\
   &\lesssim\;\|w\|_{L^{p_1}+L^{p_2}}\big(\|\psi_1\|_{L^{q_1}\cap L^{q_2}}^2+\|\psi_2\|_{L^{q_1}\cap L^{q_2}}^2\big)\|\psi_1-\psi_2\|_{L^{q_1}\cap L^{q_2}} \\
   &\lesssim \;\|w\|_{L^{p_1}+L^{p_2}}\big(\|\psi_1\|_{H_{\alpha}^1}^2+\|\psi_2\|_{H_{\alpha}^1}^2\big)\|\psi_1-\psi_2\|_{H_{\alpha}^1}\,,
  \end{split}
 \]
 having used the dual of $\|\cdot\|_{L^2}\leqslant\|\cdot\|_{H^1_\alpha}$ in the first step, estimate \eqref{eq:dig} in the second, and the continuous embedding \eqref{eq.sob1} in the third. This expresses precisely the Lipschitz continuity of $g$ on $H^1_\alpha$-bounded sets, as a $H^1_\alpha(\mathbb{R}^2)\to H^{-1}_\alpha(\mathbb{R}^2)$ map.

 Concerning \textbf{(G2)}, it was shown in the proof of Lemma \ref{le:EL} that setting
 \[
  G(\psi)\;:=\;\frac{\theta}{4}\langle g(\psi),\psi\rangle_{H_{\alpha}^{-1},H_{\alpha}^1}\;=\;\frac{\theta}{4}\int_{\mathbb{R}^2} \big(w*|\psi|^2\big) |\psi|^2\,\ud x\,, \qquad
	\psi\in H_{\alpha}^1(\mathbb{R}^2)\,,
 \]
  one has $G\in\mathcal{C}^1(H^1_\alpha(\mathbb{R}^2),\mathbb{R})$ and $G'=g$ in the above sense \textbf{(G2)}.

Condition \textbf{(G3)} follows from
 \[
  \begin{split}
    &|G(\psi_1)-G(\psi_2)|\;\leqslant\;\big|\langle g(\psi_1),\psi_1-\psi_2\rangle_{H_{\alpha}^{-1},H_{\alpha}^1}\big|+\big|\langle g(\psi_1)-g(\psi_2),\psi_2\rangle_{H_{\alpha}^{-1},H_{\alpha}^1}\big| \\
    &\qquad\leqslant\;\|(w*|\psi_1^2)\psi_1\|_{L^2}\|\psi_1-\psi_2\|_{L^2} \\
    &\qquad\qquad\qquad + \|(w*|\psi_1^2)\psi_1-(w*|\psi_1^2)\psi_1\|_{L^{q_1'}\cap L^{q_2'}}\,\|\psi_2\|_{L^{q_1}\cap L^{q_2}} \\
    &\qquad\lesssim\;\|w\|_{L^{p_1}+L^{p_2}}\Big(\|\psi_1\|_{L^{q_1}\cap L^{q_2}}^3\|\psi_1-\psi_2\|_{L^2} \\
    & \qquad\qquad\qquad\qquad\quad +\big(\|\psi_1\|_{L^{q_1}\cap L^{q_2}}^2+\|\psi_1\|_{L^{q_1}\cap L^{q_2}}^2\big)\|\psi_1-\psi_2\|_{L^2}\|\psi_2\|_{L^{q_1}\cap L^{q_2}}\Big) \\
     &\qquad\lesssim\;\|w\|_{L^{p_1}+L^{p_2}}\big(\|\psi_1\|^3_{H_{\alpha}^1}+\|\psi_2\|^3_{H_{\alpha}^1}\big)\|\psi_1-\psi_2\|_{L^2}\,,
  \end{split}
 \]
 having used estimate \eqref{eq:dig} in the third inequality and the continuous embedding  \eqref{eq.sob1} in the last.

As $\langle g(\psi),\psi\rangle_{H^{-1}_\alpha,H^1_\alpha}=4\,\theta^{-1}G(\psi)\in\mathbb{R}$, condition \textbf{(G4)} is matched.

%  It is known \cite[Lemma~5.3]{Okazawa-Suzuki-Yokota-2012} that \textbf{(G4)} implies \textbf{(G5)} under the additional assumption that for any $\psi$ and any sequence $(\psi_n)_{n\in\mathbb{N}}$ in $H^1_\alpha(\mathbb{R}^2)$

 Last, concerning \textbf{(G5)}, consider any $\psi$ and any sequence $(\psi_n)_{n\in\N}$ in $H^1_\alpha(\mathbb{R}^2)$ satisfying
 \[
  \begin{cases}
   \;\;\;\;\;\,\psi_n\,\xrightarrow{n\to\infty}\,\psi & \textrm{$H^1_\alpha$-weakly,} \\
   \;g(\psi_n)\,\xrightarrow{n\to\infty}\,\phi & \textrm{$H^{-1}_\alpha$-weakly}
  \end{cases}
 \]
  for some $\phi\in H^{-1}_\alpha(\mathbb{R}^2)$\,: the goal is now to show that $\phi=g(\psi)$. In fact, it suffices to show that $g(\psi_n)\to g(\psi)$ in $\mathcal{D}'(\mathbb{R}^2)$, and combine this with $g(\psi_n)\to \phi$ $H^{-1}_\alpha$-weakly.

  To this aim, pick $\varphi \in C_c^\infty(\mathbb{R}^2)$, arbitrary, and observe that
  \[\tag{*}\label{eq:unifbdd}
   \|\psi_n\|_{L^r}\,\lesssim_r\,\|\psi\|_{H^1_\alpha}\qquad \forall n\in\mathbb{N}\,,\;\;\forall r\in[2,\infty)
  \]
  (the $H^1_\alpha$-weak convergence of $(\psi_n)_{n\in\mathbb{N}}$ implies its uniform boundedness in $H^1_\alpha(\mathbb{R}^2)$ and hence also in any $L^r(\mathbb{R}^2)$ above, owing to the uniform boundedness principle and the continuous embedding \eqref{eq.sob1}). We re-write
  \[
   \begin{split}
    & \int_{\mathbb{R}^2}\big(g(\psi_n)-g(\psi)\big)\varphi\,\ud x \\
    & \qquad =\; \int_{\mathbb{R}^2} \big(w*|\psi_n|^2\big)(\psi_n-\psi)\varphi\,\ud x+\int_{\mathbb{R}^2} \big(w*\big(|\psi_n|^2-|\psi|^2\big)\big) \psi\varphi\,\ud x \\
    & \qquad \equiv\;(\mathrm{I}_n)+(\mathrm{II}_n)\,.
   \end{split}
  \]
 Letting $p$ stand for either $p_1,p_2$, $(\mathrm{I}_n)$ is estimated by means of \eqref{eq:yoho} as
 \[
     |(\mathrm{I}_n)|\;\lesssim\;
  \begin{cases}
   \;\|w\|_{L^1}\|\psi_n\|^2_{L^{3}}\|(\psi_n-\psi)\varphi\|_{L^{3}}\,, & \textrm{if }\; p=1\,, \\
   \;\|w\|_{L^p}\|\psi_n\|^2_{L^{2}}\|(\psi_n-\psi)\varphi\|_{L^{\frac{p}{p-1}}}\,, & \textrm{if }\; p\in(1,\infty)\,.
  \end{cases}
 \]
 Using \eqref{eq:unifbdd} and the fact that $(\psi_n-\psi)\varphi$ has compact support, we thus deduce that $|(\mathrm{I}_n)|\lesssim_{p_1,p_2}\|w\|_{L^{p_1}+L^{p_2}}\|\psi\|_{H^1_\alpha}^2\|(\psi_n-\psi)\varphi\|_{L^{q}}$ for some $q\in(1,\infty)$. Then $(\mathrm{I}_n)\to 0$ owing to the compact embedding of Corollary \ref{cpt_emb}.

 As for $(\mathrm{II}_n)$, it is convenient to re-write it as
 \[
  \begin{split}
     (\mathrm{II}_n)&\;=\;\int_{\mathbb{R}^2}\ud y \,\big(|\psi_n(y)|^2-|\psi(y)|^2\big)\,\mathcal{F}(y)\,, \\
  \mathcal{F}(y)&\;:=\;\int_{\mathbb{R}^2}\ud x\,w(x-y)\,\psi(x)\,\varphi(x)\;=\;(w*(\psi\varphi))(y)\,,
  \end{split}
 \]
 where we used that $w$ is spherically symmetric, whence even. Setting $q:=\frac{2p}{2p-1}\in[1,2)$, we see that $\mathcal{F}\in L^{q'}(\mathbb{R}^2)$: indeed,
 \[
  \begin{split}
    \|\mathcal{F}\|_{L^{q'}} \;&=\; \|w*(\psi\varphi)\|_{L^{2p}}\;\lesssim_p\;\|w\|_{L^p}\|\psi\varphi\|_{L^{q}}\;\lesssim_{p,\varphi}\|w\|_{L^p}\|\psi\|_{L^{2q}} \\
    &\lesssim_{p,\varphi}\|w\|_{L^p}\|\psi\|_{H^1_\alpha}\;<\;+\infty\,,
  \end{split}
 \]
 having used the Young inequality in the second step and the continuous embedding \eqref{eq.sob1} in the last. Therefore, we can conclude that $(\mathrm{II}_n)\to 0$ if we show that $|\psi_n|^2\to|\psi|^2$ $L^q$-weakly.

 Now, \eqref{eq:unifbdd} and $\||\psi_n|^2\|_{L^q}=\|\psi_n\|_{L^{2q}}^2$ imply that $(|\psi_n|^2)_{n\in\mathbb{N}}$ converges $L^q$-weakly, up to extracting a subsequence that we shall denote again by $(|\psi_n|^2)_{n\in\mathbb{N}}$. On the other hand, for arbitrary compact $K\subset\mathbb{R}^2$ one has
 \[
  \big\| |\psi_n|^2-|\psi|^2\big\|_{L^1(K)}\;\leqslant\;\|\psi_n-\psi\|_{L^2((K)}\big( \|\psi_n\|_{L^2(K)}+\|\psi\|_{L^2(K)}\big)\;\xrightarrow{n\to\infty}\;0\,,
 \]
 by \eqref{eq:unifbdd} and compact embedding (Corollary \ref{cpt_emb}). Hence, $|\psi_n|^2\to|\psi|^2$ in $L^1_{\mathrm{loc}}(\mathbb{R}^2)$. Then necessarily the weak-$L^q$-limit of the $|\psi_n|^2$'s is indeed $|\psi|^2$, independent of selecting a subsequence.

 Having shown that both $(\mathrm{I}_n)\to 0$ and $(\mathrm{II}_n)\to 0$, we can finally conclude that $g(\psi_n)\to g(\psi)$ in $\mathcal{D}'(\mathbb{R}^2)$: property \textbf{(G5)} is thus established.
\end{proof}

 \begin{remark}
  Let us observe that \textbf{(G1)} implies, in particular (for $\psi_2\equiv 0$),
  \begin{equation}\label{eq:esc}
   \big\|(w*|\psi|^2)\psi\big\|_{H^{-1}_\alpha}\;\leqslant\;\|\psi\|^3_{H^1_\alpha}\qquad \forall \psi\in H^1_\alpha(\mathbb{R}^2)\,,
  \end{equation}
 i.e., the Hartree non-linearity is \emph{energy sub-critical}.
 \end{remark}

\section{Enhanced local solution theory in the mass sub-critical regime}

 We detour along this Section from the main line of the proof of Theorem \ref{th:main_ex}, in order to examine more closely the mass sub-critical regime $p>1$, which includes in particular the Riesz convolution potential $w(x):=|x|^{-\eta}$, $\eta\in (0,2)$. We find that the solution map for the initial value problem \eqref{cp:sh} is locally uniformly continuous. More precisely, we obtain the following result.

 \begin{proposition}\label{pr:luco}
	Let $\alpha\in\mathbb{R}$ and $\theta=\pm 1$. Assume that $w$ satisfy \eqref{ass_w} for $p:=\min\{p_1,p_2\}>1$ (i.e., a mass sub-critical non-linearity). Then for every $M>0$ there exists $T\equiv T(M)>0$ such that the following holds: for every initial data $\psi_0,\widetilde{\psi}_0\in H_{\alpha}^1(\R^2)$, with $\|\psi_0\|_{H_{\alpha}^1},\|\widetilde{\psi}_0\|_{H_{\alpha}^1}\leqslant M$, the corresponding (maximal) solutions $\psi,\widetilde{\psi}$ to \eqref{cp:sh} (provided by Proposition \ref{pr:local_wp}) are defined on $[0,T]$, and satisfy the estimate
$$\|\psi-\widetilde{\psi}\|_{L_T^{\infty}H_{\alpha}^1}\leqslant 2\,\|\psi_0-\widetilde{\psi}_0\|_{H_{\alpha}^1}.$$
\end{proposition}

We start by collecting some useful properties of the Green function of the Laplace operator, which supplement those in Lemma \ref{lem:gom1}. To this aim, recall that each Bessel functions $K_{\nu}$, $\nu>0$ \cite[Chapter 9]{Abramowitz-Stegun-1964} is exponentially decreasing in its argument and satisfies the asymptotics \cite[Eq.~(9.6.9)]{Abramowitz-Stegun-1964}
\begin{equation}\label{asyK}
	K_{\nu}(\rho)\:\stackrel{\rho\downarrow 0}{=}\: 2^{-(1+\nu)}\,\Gamma(\nu)\,\rho^{-\nu}\,.
\end{equation}

\begin{lemma}\label{lem:gom2}
	For any $ \omega >0$ one has
	\begin{eqnarray}
		\mathcal{G}_{\omega} \!\!&\in& \!\! W^{1,r}(\mathbb{R}^2)\qquad \forall r \in [1,2)\,, \label{eq.fsp6} \\
		|x|^{\varepsilon}\,\nabla\mathcal{G}_{\omega} \!\!&\in& \!\! L^2(\mathbb{R}^2)\qquad\quad \forall \varepsilon >0\,. \label{eq.fsp7}
	\end{eqnarray}
\end{lemma}

\begin{proof}
Using the identity $\mathcal{F}\big((\omega-\Delta)^{\frac12}\,\mathcal{G}_{\omega}\big)(\xi)=(2\pi)^{-1}(\omega+|\xi|^2)^{-\frac{1}{2}}$, together with \cite[Eq.~(4.1) and (4.6)]{Aronszajn-Smith-I}, we obtain
\begin{equation*}\tag{*}\label{eqbm.1}
	 \big((\omega-\Delta)^{\frac{1}{2}}\mathcal{G}_{\omega}\big)(x)\;=\;\frac{\sqrt{\omega/2}}{\,\Gamma(\frac{1}{2})\pi\,}\,K_{\frac{1}{2}}(\sqrt{\omega}\,|x|)|\sqrt{\omega}\,x|^{-\frac12}.
\end{equation*}
Moreover, using \eqref{eq.i1} and the identity $K'_0=-K_1$ \cite[Eq.~(9.6.27)]{Abramowitz-Stegun-1964}, we get
\begin{equation*}\tag{**}\label{eqbm.2}
	(\nabla\mathcal{G}_{\omega})(x)\;=\;-\frac{\sqrt{\omega}}{2\pi}\,K_1(\sqrt{\omega}|x|)\,\frac{x}{|x|}\,.
\end{equation*}
Identities \eqref{eqbm.1} and \eqref{eqbm.2}, together with the asymptotics \eqref{asyK} and the fact that $K_{1}$ and $K_{\frac{1}{2}}$ are exponentially decreasing, yield
\begin{equation*}
\begin{split}
\|\mathcal{G}_{\omega}\|_{W^{1,r}}^r\;\approx\; \big\|(\lambda-\Delta)^{\frac{1}{2}}\mathcal{G}_{\omega}\big\|_{L^r}^r\;&\lesssim\; \int_{\substack{x\in\mathbb{R}^2 \\  |x|\leqslant 1}}\frac{\ud x}{\,|x|^{r}}\;\lesssim_r\: 1\,, \\
\||x|^{\varepsilon}\,\nabla\mathcal{G}_{\omega}\|_{L^2}^2\;&\lesssim\; \int_{\substack{x\in\mathbb{R}^2 \\  |x|\leqslant 1}}|x|^{2(\varepsilon-1)}\,\ud x\;\lesssim_{\varepsilon}\: 1
\end{split}
\end{equation*}
for every $r\in[1,2)$ and $\varepsilon>0$, which proves \eqref{eq.fsp6}-\eqref{eq.fsp7}.
\end{proof}

Next, we establish a convenient tri-linear bound, in the same spirit of \cite[Eq.~(4.3)]{MOS-SingularHartree-2017}.

\begin{proposition}\label{pr:trili}
Let $\alpha\in\mathbb{R}$ and assume that $w$ satisfy \eqref{ass_w} for $p:=\min\{p_1,p_2\}>1$. Then, for every $\psi_1,\psi_2,\psi_3\in H^1_{\alpha}(\mathbb{R}^2)$,
	\begin{equation}\label{trilinear}
		\|(w\ast(\psi_1\psi_2))\psi_3\|_{H^1_{\alpha}}\;\lesssim\;\prod_{j=1}^3\|\psi_j\|_{H^1_{\alpha}}\,.
	\end{equation}
\end{proposition}

\begin{proof}
 It is not restrictive to directly assume $w\in L^p(\R^2)$. By assumption, for $j\in\{1,2,3\}$ one has $\psi_j=f_j+c_j\mathcal{G}_{\lambda}$ for some $f_j\in H^1(\R^2)$, $c_j\in \C$, and $\lambda>|e_{\alpha}|$, and in view of \eqref{eq:equiv-norm}  $\|\psi_j\|_{H_{\alpha}^1}\approx\|f_j\|_{H^1}+|c_j|$. For $h:=w*(\psi_1\psi_2)$ we find
\begin{equation*}\tag{a}\label{pgpg}
	\begin{split}
\|h\|_{W^{1,2p}}\;&\lesssim\; \|w*\nabla(\psi_1\psi_2)\|_{L^{2p}}\;\lesssim\; \|w\|_{L^p}\|\nabla(\psi_1\psi_2)\|_{L^{\frac{2p}{2p-1}}}\\
&\lesssim \;\|f_1f_2+(c_2f_1+c_1f_2)\mathcal{G}_{\lambda}+c_1c_2\mathcal{G}_{\lambda}^2\|_{W^{1,\frac{2p}{2p-1}}}\\
&\lesssim\;\|f_1\|_{H^1}\|f_2\|_{L^{\frac{2p}{p-1}}}+\|f_1\|_{L^{\frac{2p}{p-1}}}\|f_2\|_{H^1}\\
&\qquad +|c_2|\big(\|f_1\|_{H^1}\|\mathcal{G}_{\lambda}\|_{L^{\frac{2p}{p-1}}}+\|f_1\|_{L^{\frac{4p}{p-1}}}\|\mathcal{G}_{\lambda}\|_{W^{1,\frac{4p}{3p-1}}}\big)\\
&\qquad +|c_1|\big(\|f_2\|_{H^1}\|\mathcal{G}_{\lambda}\|_{L^{\frac{2p}{p-1}}}+\|f_2\|_{L^{\frac{4p}{p-1}}}\|\mathcal{G}_{\lambda}\|_{W^{1,\frac{4p}{3p-1}}}\big)\\
&\qquad +|c_1||c_2|\|\mathcal{G}_{\lambda}\|_{L^{\frac{4p}{p-1}}}\|\mathcal{G}_{\lambda}\|_{W^{1,\frac{4p}{3p-1}}},
\end{split}
\end{equation*}
where we used the Young inequality in the second step and the Leibniz rule together with H\"older inequality in the last step. Observe moreover that
\begin{equation*}\tag{b}\label{gpgpgp}
\mathcal{G}_{\lambda}\;\in\; L^{\frac{4p}{p-1}}(\R^2)\cap W^{1,\frac{4p}{3p-1}}(\R^2)\,,
	\end{equation*}
as follows from \eqref{eq.fsp1}, \eqref{eq.fsp6} and the assumption $p>1$. Using \eqref{pgpg}, the Sobolev embedding, and \eqref{gpgpgp}, we deduce
\begin{equation*}\tag{c}\label{compg1}
\|h\|_{W^{1,2p}}\;\lesssim \;\big(\|f_1\|_{H^1}+|c_1|\big)\big(\|f_2\|_{H^1}+|c_2|\big)\;\lesssim\; \|\psi_1\|_{H_{\alpha}^1}\|\psi_2\|_{H_{\alpha}^1}\,.
\end{equation*}
Next, we split
$$(w*(\psi_1 \psi_2))\psi_3\;=\;\phi+c_3\,h(0)\,\mathcal{G}_{\lambda}\qquad\textrm{ with }\qquad \phi:=hf_3+(h-h(0))\,c_3\,\mathcal{G}_{\lambda}\,.$$
By means of the Sobolev embedding ($W^{1,2p}(\R^2)\hookrightarrow L^{\infty}(\R^2)$, owing to the assumption $p>1$), we obtain
\begin{equation*}\tag{d}\label{fr}
\|hf_{3}\|_{H^1} \;\lesssim\; \|h\|_{W^{1,2p}}\|f_{3}\|_{L^{2p/(p-1)}}+\|h\|_{L^{\infty}}\|f_{3}\|_{H^1}\;\lesssim\;\|h\|_{W^{1,2p}}\|f_{3}\|_{H^1}
\end{equation*}
and
\begin{equation*}\tag{e} \label{tr}
 \|(h-h(0))\mathcal{G}_{\lambda}\|_{L^2}\;\lesssim\; \|h\|_{L^{\infty}}\|\mathcal{G}_{\lambda}\|_{L^2}\lesssim\|h\|_{W^{1,2p}}\,.
\end{equation*}
On the other hand, the Morrey-Sobolev embedding guarantees that $W^{1,2p}(\R^2)\hookrightarrow\mathcal{C}^{0,\varepsilon}(\R^2)$ for some $\varepsilon\equiv\varepsilon(p)>0$. Using \eqref{eq.fsp1} and \eqref{eq.fsp6} we then obtain
\begin{equation*}\tag{f}\label{pr}
	\begin{split}
		\|\nabla\big((h-h(0))\mathcal{G}_{\omega}\big)\|_{L^2}\;&\lesssim\; \|\nabla h\|_{L^{2p}}\|\mathcal{G}_{\omega}\|_{L^{2p/(p-1)}}\\
		&\qquad+\||x|^{-\varepsilon}(h-h(0))\|_{L^{\infty}}\||x|^{\varepsilon}\,\nabla\mathcal{G}_{\omega}\|_{L^2} \\
		&\lesssim\;\|h\|_{W^{1,2p}}\,.
	\end{split}
\end{equation*}
Combining \eqref{fr}, \eqref{tr}, and \eqref{pr} we deduce that $\phi\in H^1(\R^2)$, with
\begin{equation*}\tag{g}\label{estiphi}
\|\phi\|_{H^1}\;\lesssim\;\|h\|_{W^{1,2p}}\big(\|f_{3}\|_{H^1}+|c_3|\big)\;\lesssim\; \|h\|_{W^{1,2p}}\|\psi_3\|_{H_{\alpha}^1}\,.
\end{equation*}
Therefore, $(w*(\psi_1\psi_2))\psi_3=\phi+c_3\,h(0)\,\mathcal{G}_{\lambda}\in H^1(\R^2)\oplus\operatorname{Span}(\mathcal{G}_{\omega})\cong H_{\alpha}^1(\R^2)$. Moreover, combining \eqref{compg1} and \eqref{estiphi} we eventually obtain
$$\|(w*(\psi_1\psi_2))\psi_3\|_{H_{\alpha}^1}\;\lesssim\;\|\phi\|_{H^1}+|c_3|\|h\|_{L^{\infty}}\;\lesssim\;\|h\|_{W^{1,2p}}\|\psi_3\|_{H_{\alpha}^1}\;\lesssim\; \prod_{j=1}^3\|\psi_j\|_{H^1_{\alpha}}\,,$$
that is, the estimate \eqref{trilinear}.
\end{proof}

As a consequence of Proposition \ref{pr:trili}, it turns out that the Hartree non-linearity is locally Lipschitz in the energy space.

\begin{corollary}
Let $\alpha\in\mathbb{R}$ and assume that $w$ satisfy \eqref{ass_w} for $p:=\min\{p_1,p_2\}>1$. Then, setting $g(\psi):=(w*|\psi|^2)\psi$,
\begin{equation}\label{llhn}
	\|g(\psi_1)-g(\psi_2)\|_{H_{\alpha}^1}\;\lesssim\; M^2\,\|\psi_1-\psi_2\|_{H_{\alpha}^1}
\end{equation}
for every $\psi_1,\psi_2\in H_{\alpha}^1(\R^2)$ with $\|\psi_1\|_{H_{\alpha}^1},\|\psi_1\|_{H_{\alpha}^1}\leqslant M$.
\end{corollary}

\begin{proof}
From the identity
	\begin{equation*}%\label{wic}
\begin{split}
	g(\psi_1)-g(\psi_2)\;&=\;(w*|\psi_1|^2)(\psi_1-\psi_2)+\big(w*\big(|\psi_1|^2-|\psi_2|^2\big)\big)\psi_2 \\
	&=\;(w*(v_1\overline{v_1}))(\psi_1-\psi_2)+\big(w*\big(v_1(\overline{\psi_1-\psi_2})+\overline{\psi_2}(\psi_1-\psi_2)\big)\big)\psi_2\,,
\end{split}
\end{equation*}
 the tri-linear bound \eqref{trilinear} and the identity $\|v\|_{H_{\alpha}^1}=\|\overline{v}\|_{H_{\alpha}^1}$ (which follows from the fact that $-\Delta_{\alpha}$ is a real operator), yield \eqref{llhn}.
\end{proof}

Since the Hartree non-linearity is locally Lipschitz in the energy space, the standard Kato method is applicable to the initial value problem \eqref{cp:sh} without needing the knowledge of the adapted space $W_{\alpha}^{1,r}(\R^2)$ for $r\neq 2$: this produces eventually Proposition \ref{pr:luco}. For concreteness, we present here a direct proof.

\begin{proof}[Proof of Proposition \ref{pr:luco}]
For every $T<\max\{T_{\mathrm{max}}(\psi_0),T_{\mathrm{max}}(\widetilde{\psi}_0)\}$, the unitarity of the propagator $e^{it\Delta_{\alpha}}$ in $H_{\alpha}^1(\R^2)$ and the tri-linear estimate \eqref{trilinear} yield
\begin{equation*}
\begin{split}
\|\psi\|_{L_T^{\infty}H_{\alpha}^1}\;&\leqslant\; \|\psi_0\|_{H_{\alpha}^1}+\bigg\|\int_0^t (w*|\psi(\tau,\cdot)|^2)|\psi(\tau,\cdot)|\,\ud \tau\,\bigg\|_{L_T^{\infty}H_{\alpha}^1}\\
&\leqslant\; M+T\,\|(w*|\psi|^2)\psi\|_{L_T^{\infty}H_{\alpha}^1}\;\leqslant\; M+CT\|\psi\|^3_{L_T^{\infty}H_{\alpha}^1}
\end{split}
\end{equation*}
for a suitable constant $C>0$, and the analogous estimate for $\widetilde{\psi}$. Then, for $T^*:=\frac12(8CM^2)^{-1}$,
\begin{equation}\label{eq:contraction}
	\|\psi\|_{L_{T^*}^{\infty}H_{\alpha}^1}\;\leqslant\; 2M\,,\qquad \|\widetilde{\psi}\|_{L_{T^*}^{\infty}H_{\alpha}^1}\;\leqslant\; 2M\,.
	\end{equation}	
In view of the blow-up alternative, $T^*<\max\{T_{\mathrm{max}}(\psi_0),T_{\mathrm{max}}(\widetilde{\psi}_0)\}$. Moreover, for every $T\in(0,T^*)$,
\begin{equation*}
\begin{split}
\|\psi-\widetilde{\psi}\|_{L_T^{\infty}H_{\alpha}^1}\;&\leqslant\; \|\psi_0-\widetilde{\psi}_0\|_{H_{\alpha}^1}+ \\
& \quad +\bigg\|\int_0^t\Big((w*|\psi(\tau,\cdot)|^2)\psi(\tau,\cdot)-(w*|\widetilde{\psi}(\tau,\cdot)|^2)\widetilde{\psi}(\tau,\cdot)\Big)\,\ud\tau\bigg\|_{L_T^{\infty}H_{\alpha}^1}\\
&\leqslant\; \|\psi_0-\widetilde{\psi}_0\|_{H_{\alpha}^1}+DT\|\psi\|_{L_T^{\infty}H_{\alpha}^1}\|\widetilde{\psi}\|_{L_T^{\infty}H_{\alpha}^1}\|\psi-\widetilde{\psi}\|_{L_T^{\infty}H_{\alpha}^1}\\
&\leqslant\; \|\psi_0-\widetilde{\psi}_0\|_{H_{\alpha}^1}+4DTM^2\|\psi-\widetilde{\psi}\|_{L_T^{\infty}H_{\alpha}^1}
\end{split}
\end{equation*}
for a suitable constant $D>0$, having used the unitarity of $e^{it\Delta_{\alpha}}$ in $H_{\alpha}^1(\R^2)$ in the first step,  \eqref{llhn} in the second, and \eqref{eq:contraction} in the last. The thesis then follows by choosing $T(M):=\min\{T^*,(8DM^2)^{-1}\}$.
\end{proof}

\section{Global well-posedness in the energy space}\label{sec:GWP}

We complete in this Section the proof of Theorem \ref{th:main_ex} by extending the local analysis of Section \ref{sec:LWP} globally in time. Let us stress here the different reasoning needed depending on whether the non-linearity is focusing or defocusing.

%In this final section we prove the global well-posedness result for the singular Hartree equation \eqref{singular_hartree-theta}, stated in Theorem \ref{th:main_ex}.

%Preliminary we prove the following Gagliardo-Nirenberg type inequality for the Hartree non-linearity, which will be crucial for extending globally-in-time the solutions in the focusing, mass-critical case.

%Now we are ready to prove the global well-posedness result.

\begin{proof}[Proof of Theorem \ref{th:main_ex}]
 As already mentioned, the local well-posedness of the initial value problem \eqref{cp:shmain} in the energy space $H^1_\alpha(\mathbb{R}^2)$ is established in Proposition \ref{pr:local_wp}.

 In the \emph{defocusing} regime $\theta=1$ the unique local solution $\psi\in\mathcal{C}([0,T_{\mathrm{max}}),H^1_\alpha(\mathbb{R}^2))$ provided by Proposition \ref{pr:local_wp} is also global in time as a direct consequence of the blow-up alternative and the conservation of mass and energy. Indeed,
 \[
  \begin{split}
   \|\psi(t,\cdot)\|^2_{H^1_\alpha}\;&=\;(-\Delta_\alpha)[\psi(t,\cdot)]+(1-e_\alpha)\|\psi(t,\cdot)\|_{L^2}^2 \\
   &\lesssim_\alpha\,\mathcal{E}(t)+\mathcal{M}(t)\;=\;\mathcal{E}(0)+\mathcal{M}(0)\qquad \forall t\in[0,T_{\mathrm{max}})
  \end{split}
 \]
 (having applied \eqref{eq:energynormH1alpha} in the first step and Proposition \ref{pr:local_wp}(ii) in the third, and having used the positivity of $\theta$ in the second), whence $T_{\mathrm{max}}=+\infty$ on account of Proposition \ref{pr:local_wp}(i), and $\sup_{t>0}\|\psi(t,\cdot)\|_{H^1_\alpha}<+\infty$.

 The next case to consider is the \emph{focusing} ($\theta=-1$) and \emph{mass sub-critical} ($p=\mathrm{min}\{p_1,p_2\}>1$) regime. Now one has
  \[
  \begin{split}
   \|\psi(t,\cdot)\|^2_{H^1_\alpha}\;&=\;(-\Delta_\alpha)[\psi(t,\cdot)]+(1-e_\alpha)\|\psi(t,\cdot)\|_{L^2}^2  \\
      &\lesssim_\alpha\,\mathcal{E}(t)+\mathcal{M}(t)+	\int_{\mathbb{R}^2}\big(w*|\psi(t,\cdot)|^2\big)(x)|\psi(t,x)|^2\,\ud x   \\
      &=\;\mathcal{E}(0)+\mathcal{M}(0)+\int_{\mathbb{R}^2}\big(w*|\psi(t,\cdot)|^2\big)(x)|\psi(t,x)|^2\,\ud x \qquad \forall t\in[0,T_{\mathrm{max}})\,.
  \end{split}
 \]
  The only $t$-dependent summand in the r.h.s.~above is controlled, by means of \eqref{eq:gani} and again the mass conservation, so that
  \[
   \|\psi(t,\cdot)\|^2_{H^1_\alpha}\;\lesssim_{\alpha,p}\, 1+\|\psi(t,\cdot)\|^{\frac{2}{p}}_{H^1_\alpha}\qquad \forall t\in[0,T_{\mathrm{max}})\,.
  \]
 Since $\frac{2}{p}<2$, the above estimate prevents $\|\psi(t,\cdot)\|_{H^1_\alpha}$ to blow up in finite time, whence again $T_{\mathrm{max}}=+\infty$ owing to Proposition \ref{pr:local_wp}(i), and $\sup_{t>0}\|\psi(t,\cdot)\|_{H^1_\alpha}<+\infty$.

  The last regime to consider is $\theta=-1$, $p=1$ (focusing and \emph{mass critical}). In this case \eqref{eq:gani} takes the form
  \[
   \int_{\mathbb{R}^2}\big(w*|\psi(t,\cdot)|^2\big)(x)|\psi(t,x)|^2\,\ud x\;\leqslant\;C_{\textsc{gn}}\|\psi(t,\cdot)\|_{H_{\alpha}^1}^{2} \|\psi(t,\cdot)\|_{L^2}^{2}\,,
  \]
 having indicated by $C_{\textsc{gn}}\equiv C_{\textsc{gn}}(\alpha,\|w\|_{L^{p_1}+L^{p_2}})>0$ the \emph{optimal} constant of such a Gagliardo-Nirenberg type inequality. Thus now
  \[
  \begin{split}
   \|\psi(t,\cdot)\|^2_{H^1_\alpha}\;&=\;(-\Delta_\alpha)[\psi(t,\cdot)]+(1-e_\alpha)\|\psi(t,\cdot)\|_{L^2}^2  \\
      &=\;2\mathcal{E}(t)+\frac{1}{2}\int_{\mathbb{R}^2}\big(w*|\psi(t,\cdot)|^2\big)(x)|\psi(t,x)|^2\,\ud x + (1-e_\alpha)\mathcal{M}(t) \\
      &\leqslant\;\big(2\mathcal{E}(t)+(1-e_\alpha)\mathcal{M}(t)\big)+\frac{1}{2}\,C_{\textsc{gn}}\,\mathcal{M}(t) \|\psi(t,\cdot)\|^2_{H^1_\alpha}\qquad \forall t\in[0,T_{\mathrm{max}})\,.
  \end{split}
 \]
 Introducing the \emph{additional assumption} $\mathcal{M}(0)<2\,C_{\textsc{gn}}^{-1}$ and exploiting mass and energy conservation, we deduce from the latter estimate that
 \[
  \|\psi(t)\|_{H_{\alpha}^1}^2\;\leqslant\; \big(1-{\textstyle{\frac{1}{2}}}C_{\textsc{gn}}\mathcal{M}(0)\big)^{-1}\big(2\mathcal{E}(0)+(1-e_{\alpha})\mathcal{M}(0)\big)\qquad \forall t\in[0,T_{\mathrm{max}})\,,
 \]
 whence $T_{\mathrm{max}}=+\infty$ owing to Proposition \ref{pr:local_wp}(i), and $\sup_{t>0}\|\psi(t,\cdot)\|_{H^1_\alpha}<+\infty$.
\end{proof}

% \bibliographystyle{siam}
% \bibliography{bib_ALE}

\def\cprime{$'$}

\end{document}